\documentclass[11pt]{amsart}

\usepackage{amsmath,leftidx,amsthm,sfheaders}
\usepackage{xspace}
\usepackage[psamsfonts]{amssymb}
\usepackage[latin1]{inputenc}
\usepackage{graphicx,color}
\usepackage[all]{xy}

\usepackage{hyperref}

\theoremstyle{remark}

\newtheorem*{remark}{\bf Remark}

\theoremstyle{plain}

\newtheorem{theorem}{\bf Theorem}[section]
\newtheorem{proposition}[theorem]{\bf Proposition}
\newtheorem{definition}[theorem]{\bf Definition}
\newtheorem{Theorem}{\bf Theorem}

\newtheorem{lemma}[theorem]{\bf Lemma}
\newtheorem{corollary}[theorem]{\bf Corollary}

\def\A{{\mathbb A}}
\def\C{{\mathbb C}}
\def\R{{\mathbb R}}
\def\Z{{\mathbb Z}}

\def\Q{{\mathbb Q}}

\def\D{\mathbb{D}}
\def\KK{\mathbb{K}}
\def\p{\mathbb{P}}
\def\N{{\mathbb N}}

\def\cP{{\mathcal P}}
\def\O{\mathcal{O}}

\def\tC{\tilde{C}}

\def\fX{\mathfrak{X}}
\def\fL{\mathfrak{L}}

\def\um{\underline{m}}
\def\un{\underline{n}}
\def\uw{\underline{w}}

\def\om{\omega}
\def\e{\varepsilon}
\def\a{\alpha}

\def\pe{\textup{ := }}
\def\an{\textup{an}}

\DeclareMathOperator{\Aut}{Aut}

\DeclareMathOperator{\id}{id}
\DeclareMathOperator{\ord}{ord}

\DeclareMathOperator{\reg}{Reg}
\DeclareMathOperator{\sing}{Sing}

\DeclareMathOperator{\Per}{Per}
\DeclareMathOperator{\PCF}{PCF}

\DeclareMathOperator{\TPCF}{TPCF}

\DeclareMathOperator{\area}{Area}

\def\nv{\textup{nv}}
\def\bif{\textup{bif}}
\def\ingram{\textup{ingram}}

\DeclareMathOperator{\poly}{Poly}

\DeclareMathOperator{\spec} {Spec}
\DeclareMathOperator{\card} {Card}

\def\and{{\quad\text{and}\quad}}

\begin{document}

\title{Distribution of postcritically finite polynomials}
\author{C. Favre \and T. Gauthier}
\address{Centre de Math\'ematiques Laurent Schwartz, \'Ecole Polytechnique, 91128 Palaiseau Cedex, France}
\email{favre@math.polytechnique.fr}
\address{LAMFA UMR-CNRS 7352, Universit\'e de Picardie Jules Verne, 33 rue Saint-Leu, 80039 Amiens Cedex 1, France}
\address{Institute for Mathematical Sciences, Stony Brook University, 100 Nicolls Rd, Stony Brook, NY 11794,  USA}
\email{thomas.gauthier@u-picardie.fr}

\thanks{First author is supported by the ERC-starting grant project "Nonarcomp" no.307856}
\maketitle

\begin{abstract}
We prove that  Misiurewicz parameters with prescribed combinatorics and hyperbolic parameters
with $(d-1)$ distinct attracting cycles with given multipliers are equidistributed with respect to the bifurcation measure in the moduli space of degree $d$ complex polynomials. Our proof relies on Yuan's equidistribution results of points of small heights, and uses in a crucial way Epstein's transversality results.
\end{abstract}

\tableofcontents


\section*{Introduction.}
DeMarco \cite{DeMarco1}  proved that, in any holomorphic family of rational maps, the bifurcation locus is the support of a closed positive $(1,1)$-current with continuous potential. In the space $\mathcal{P}_d$  of   complex polynomials (resp. rational maps) of degree $d$ modulo conjugacy by
affine (resp. M\"obius) transformations, this current induces a \emph{bifurcation measure}
which may be seen in many  ways as the analogue of the harmonic measure of the Mandelbrot set when $d\ge3$. This measure was first introduced by Bassanelli and Berteloot \cite{BB1}, and detects maximal bifurcation phenomena. Its support is known to coincide with the closure of strictly postcritically finite parameters (see \cite{favredujardin,buffepstein,Article2}) and with the closure of parameters having a maximal number of neutral cycles by~\cite{BB1}. It is also known to have total Hausdorff dimension~\cite{Article1}.
\par Our main goal  is to prove the equidistribution of postcritically finite parameters to the bifurcation measure in  $\mathcal{P}_d$  for any $d\ge 2$. Our technique builds on Ingram's observation~\cite{Ingram} that $\mathcal{P}_d$  carries a natural height function, and our result ultimately follows from Yuan's equidistribution result of points of small heights~\cite{Yuan}. 
We note that the idea of using Yuan's result for studying problems on parameter spaces of higher dimension also appeared in a recent paper by Ghioca, H'sia and Tucker~\cite{GHT}.

~

\par We  work  with the following  ``orbifold" parameterization of $\cP_d$, see~\cite{BH}. For $(c,a)=(c_1\ldots,c_{d-2},a)\in\C^{d-1}$, we set
\begin{equation}\label{eq:defpoly}
P_{c,a}(z)\pe \frac{1}{d}z^d+\sum_{j=2}^{d-1}(-1)^{d-j}\sigma_{d-j}(c)\frac{z^j}{j}+a^d,
\end{equation}
where $\sigma_j(c)$ is the monic symmetric polynomial in $(c_1,\ldots,c_{d-2})$ of degree $j$.
Observe that the critical points of $P_{c,a}$ are exactly $c_0$, $c_1, \ldots , c_{d-2}$ with the convention that $c_0=0$, and that the canonical projection $\pi:\C^{d-1}\longrightarrow\mathcal{P}_d$  which maps $(c_1,\ldots,c_{d-2},a) \in\C^{d-1}$ to the class of $P_{c,a}$ in $\mathcal{P}_d$ is finite-to-one.

Recall that the \emph{Green function} of a polynomial $P_{c,a}$ is a continuous subharmonic function defined as the following (uniform) limit:
\begin{eqnarray*}
g_{c,a} (z) \pe\lim_{n\rightarrow\infty} d^{-n} \log^+|P^n_{c,a}(z)|~,
\end{eqnarray*}
where $\log^+$ stands for  $\max \{0, \log\}$. The Julia set of $P_{c,a}$ is characterized by the equality $\mathcal{J}(P_{c,a}) = \partial\{ g_{c,a} = 0 \}$.
One can show~\cite[\S 6]{favredujardin} that the function
\begin{center}
$G(c,a)\pe\max\{g_{c,a}(c_0),g_{c,a}(c_1),\ldots,g_{c,a}(c_{d-2})\}$
\end{center}
is a continuous plurisubharmonic (psh) function on $\C^{d-1}$. 
The bifurcation measure is by definition the Monge-Amp\`ere measure of $G$, that is
$\mu_\bif\pe(dd^cG)^{d-1}$. Its support is compact and  
coincides with the Shilov boundary of the connectedness locus $\mathcal{C}_d=\{(c,a)\in\C^{d-1}, \, \mathcal{J}(P_{c,a}) \text{ is connected}\}$. It is contained in the set of parameters at which all critical points bifurcate.

Our first result can be stated as follows.
\begin{Theorem}
For all  $0\leq i\leq d-2$,  pick a sequence of integers $m_{k,i}> n_{k,i}>0$ such that $m_{k,i} - n_{k,i} \rightarrow \infty$ as $k\rightarrow\infty$. 
Consider the probability measure $\mu_k$ that is uniformly distributed over all parameters $(c,a)\in \C^{d-1}$ s.t. $P^{n}(c_i)$ is periodic iff $n\ge n_{k,i}$ and its period
is exactly equal to $m_{k,i} - n_{k,i}$. 

Then the measures $\mu_k$ converge in the weak sense to $\mu_\bif$ as $k\rightarrow\infty$.
\label{maintm1}
\end{Theorem}
In plain words, this theorem says that strictly postcritically finite polynomials with prescribed combinatorics
are equidistributed with respect to the bifurcation measure. 
Our second result deals with postcritically finite hyperbolic polynomials.
\begin{Theorem}
For each $0\leq i\leq d-2$ choose a sequence of integers $m_{k,i}$ such that $m_{k,i} \rightarrow \infty$ as $k\rightarrow\infty$ with $m_{k,i}\ne m_{k,j}$ for all $i\ne j$ and all $k$. 
Consider the probability measure $\mu'_k$ that is uniformly distributed over all parameters $(c,a)\in \C^{d-1}$ such that $c_1, \ldots c_{d-1}$ are periodic of respective periods exactly $m_{k,1}, \ldots , m_{k,d-1}$ respectively.

Then the measures $\mu'_k$ converge in the weak sense to $\mu_\bif$ as $k\rightarrow\infty$.
\label{maincrittm3}
\end{Theorem}

For quadratic polynomials, this result goes back to Levin~\cite{Levin1} (see also \cite{Levin4}). A proof of Levin's result based on a one-dimensional equidistribution result of points of small height was later given by Baker and H'sia~\cite{baker-hsia}.
Estimates for the speed of convergence of $\mu'_k$ to $\mu_\bif$ were then obtained by the first author and Rivera-Letelier~\cite{FRL} using the idea of Baker and H'sia.

\medskip

The set of parameters $(c,a)\in\C^{d-1}$ for which $P_{c,a}$ has a periodic point of exact period $n$ and multiplier $w$ is an algebraic hypersurface of $\C^{d-1}$ which we denote by $\Per^*(n,w)$. Bassanelli and Berteloot  have  studied the distribution of the hypersurfaces $\Per^*(n,w)$ for a fixed $w$, as $n\to\infty$. In particular, they proved in \cite{BB2} that for all $w$ in the closed unit disk, the currents $d^{-n}[\Per^*(n,w)]$ converge to the bifurcation current as $n\to\infty$. In~\cite{BB3}, they also proved that for any $r>0$ the measures
$$\frac{d^{-(k_1(n)+\cdots+k_{d-1}(n))}}{(d-1)!(2\pi)^{d-1}} \int_{[0,2\pi]^{d-1}}\bigwedge_{j=1}^{d-1}[\Per^* (k_j(n),r e^{i\theta_j})]\, d\theta_1\cdots d\theta_{d-1}$$
converge to the bifurcation measure for a suitable choice of increasing functions $k_j:\N\to\N$ (compare with \cite{bsurvey}). Inspired by the seminal work of Briend and Duval~\cite{briendduval2} on the 
construction of the measure of maximal entropy for endomorphisms of the complex projective space, we derive from   Theorem~\ref{maincrittm3} an equidistribution result for the sets
$$
\Per^*(\um,\uw)\pe\bigcap_{i=1}^{d-1}\Per^*(m_i,w_i)~,
$$
when $|w_i|< 1$ for all $i$. Observe that the union of these sets over all $\um$ is known to contain the support of the measure, see~\cite[Theorem~5.2.9]{bsurvey}.
Namely we prove 
\begin{Theorem}
For each $1\leq i\leq d-1$  pick $w_i$ in the open unit disk, and choose a sequence of integers $m_{k,i}$ such that $m_{k,i} \rightarrow \infty$ as $k\rightarrow\infty$. Assume in addition that $m_{k,i}\neq m_{k,j}$ for all $k$ and all $i\neq j$. Then the set $Per^*(\um_k,\uw_k)$ is finite, and the probability measure $\mu''_k$ that is uniformly distributed over this set is well-defined. 

Moreover the sequence $\mu''_k$ converges in the weak sense to $\mu_\bif$ as $k\rightarrow\infty$.
\label{maintm3}
\end{Theorem}

R.~Dujardin and the first author obtained in~\cite[Theorem~5]{favredujardin}
the existence of a sequence of atomic measures supported on strictly postcritically finite parameters (or critically finite hyperbolic polynomials)  and converging to the bifurcation measure. 
Theorems \ref{maintm1} and \ref{maincrittm3} are strengthening of these statements.

\medskip

As mentioned above the proofs of Theorems~\ref{maintm1} and~\ref{maincrittm3} relie on Yuan's equidistribution theorem of points of small height. 

The first problem is to construct a  height on the space of  
polynomials of degree $d$ to which Yuan's result can be applied. In technical terms one needs to
prove that the height is associated with a so-called continuous \emph{adelic semi-positive}
metrized line bundle. 
To any polynomial $P_{c,a}$ defined over $\bar{\Q}$ is associated a natural height function $h_{c,a} : \bar{\Q} \to \R_+$. A first natural height on the space of polynomials is defined by the formula $ \sum_{i=0}^{d-2} h_{c,a}(c_i)$. This height was used by P.~Ingram in op. cit., but it is not associated with a continuous metrization.
Our first observation is that the slightly modified height $ \max_{0\le i \le d-2} h_{c,a}(c_i)$ is induced by a continuous adelic semipositive metric, see Sections~\ref{sec:archi} and~\ref{sec:nonarchi}. Our estimates are very close in spirit to the ones given in~\cite{GHT}.

We then need to overcome a second issue.
Yuan's result only applies to sequences of finite collections of points $Z_k$ that are \emph{generic} in the following sense. For  any proper subvariety $V$ the proportion of points of $Z_k$ lying in $V$ is negligible, or in other words   $\lim_k \card (Z_k\cap V) / \card (Z_k)  = 0$.
Checking this condition for postcritically finite maps constitutes the core of our analysis, and occupies most of Section~\ref{sec:distrib}. To do so we rely on  transversality results describing how the hypersurfaces of parameters where one critical point is preperiodic intersect in the parameter space of polynomials with marked critical points. We thus use in an essential way the key contributions of A.~Epstein as exposed in~\cite{epstein2,buffepstein}.  In Section~\ref{sec:trans} we explain how to adapt his arguments to our situation.

We note that the recent work of Baker and DeMarco~\cite{BD}  deals with the (much) more delicate problem of characterizing those positive dimensional subvarieties $V$ in the parameter space such that $ \card (Z_k\cap V) \to \infty$. We shall not rely on their result.

Another ingredient also appears in the course of the proof of Theorem~\ref{maintm1}. Namely, 
our counting of strictly post-critically finite maps is based on the notion of critical portrait that was introduced by Fisher~\cite{Fisher} and on the continuity result of Bielefeld, Fisher and Hubbard~\cite{BFH} and Kiwi~\cite{kiwi-portrait}.

\medskip

We observe that Yuan's Theorem also yields equidistribution result at finite places. 
For any prime $p>0$, one can replace $\C$ by $\C_p$ in the statements of Theorems~\ref{maintm1} and~\ref{maincrittm3}, the completion of the algebraic closure of $\Q_p$. The corresponding atomic measures $\mu_k, \mu'_k$ are now
supported on the analytic Berkovich space associated with $\A_{\C_p}^{d-1}$, and converges
to the same probability measure $\mu_{\bif,p}$.

~

\par Our approach relies in an essential way on the compactness of the support of $\mu_\bif$ in the space of polynomials. This property is not satisfied by the support of the bifurcation measure in the space of rational maps even in degree $2$, see~\cite[Proposition 5.1]{Mod2}. The equidistribution of postcritically finite parameters in the context of rational maps is thus widely open. There is an important literature on  bifurcations of rational maps, and we refer to~\cite{bsurvey,dsurvey,McMullen} for more informations on this subject.

\medskip

\noindent {\bf Acknowledgements:} we heartfully thank X. Buff for his crucial help in the understanding of the transversality results of A.~Epstein. 


\section{At the archimedean place.}\label{sec:archi}


\subsection{Basics.}
Pick any $(c,a) \in \C^{d-1}$ and consider the degree $d$ polynomial $P_{c,a}$ as in the introduction. Recall that $g_{c,a}(z)=0$ if and only if the forward orbit $\{P^n_{c,a}(z)\}_{n\ge0}$ is bounded. The Julia set $\mathcal{J}_{c,a}$ of $P_{c,a}$ coincides with the boundary of the locus $\{g_{c,a}=0\}$. Moreover $P_{c,a}$ has a connected Julia set if and only if all its critical points have bounded forward orbits i.e. if and only if $G(c,a) \pe\max\{g_{c,a}(c_0),g_{c,a}(c_1),\ldots,g_{c,a}(c_{d-2})\}=0$. 

 Recall that a polynomial $P_{c,a}$ is \emph{Misiurewicz} if all critical points $c_0,\ldots,c_{d-2}$ are preperiodic to repelling cycles. It is \emph{hyperbolic} if all critical points eventually lands on the basin of attraction of an attracting (or super-attracting) cycle.

We shall use the following two results that are proved in~\cite[\S 4]{BB2}.
\begin{lemma}
The polynomials $P_{c,a}(c_i)\in\mathbb{Q}[c_1,\ldots,c_{d-2},a]$ are homogeneous polynomials of degree $d$ for $0\leq i\leq d-2$ with no other common roots than $(0,\ldots,0)$.
\label{lminfinity}
\end{lemma}
For simplicity we write $|c| \pe \max_{1\le i \le d-2} |c_i|$.
\begin{proposition}
One can write
\begin{eqnarray*}
G(c,a)=\log^+\max\left\{|c|,|a|\right\}+O(1)~.
\end{eqnarray*}
Moreover, for any $0\le i \le d-2$, the closure of the set
$\mathcal{B}_i\pe\{(c,a)\in \A_\C^{d-1}, \,  g_{c,a}(c_i)=0\}$ in $\p^{d-1}_\C$ is equal to $\mathcal{B}_i\cup\{[c:a:0]\in\p^{d-1}_\C, \, P_{c,a}(c_i)=0\}$.
\label{propBH}
\end{proposition}


\subsection{The  bifurcation measure.}

Let $\omega$ be the Fubini-Study form on $\p^{d-1}_\C$, defined  by $\om = \frac12 dd^c \log (1 + |a|^2 + \sum_1^{d-1} |c_i|^2)$ in $\A^{d-1}_\C$ and  normalized in such a way that $\int_{\p^{d-1}_\C}\omega^{d-1}=1$. The mass of a positive closed current $T$ of bidegree $(p,p)$ on $\p^{d-1}_\C$ is given by
\begin{eqnarray*}
\|T\|\pe\int_{\p_\C^{d-1}}T\wedge\omega^{d-1-p}.
\end{eqnarray*}
Pick any psh function $u: \mathbb{A}_\C^{d-1}\to [-\infty, +\infty)$ such that $ u (c,a ) \le \log^+\max\{|c|,|a|\} + O(1)$. Then one can show that the  current $T= dd^c u$ extends uniquely to  a  positive closed $(1,1)$ current on $\p^{d-1}_\C$ that does not charge the hyperplane at infinity $H_\infty := \p^{d-1}_\C \setminus \A_\C^{d-1}$, and satisfies $\| T\| = 1$. One can also show that if $H\subset\p^{d-1}$ is an algebraic hypersurface, then $\|[H]\|=\deg(H)$.

~

\par For any  $0\leq i\leq d-2$, the function $g_i(c,a) := g_{c,a}(c_i)$ is continuous and psh in $\A_\C^{d-1}$, and satisfies the upper bound above. In particular  the positive closed $(1,1)$ current $T_i:= dd^c g_i$ extends to $\p^{d-1}_\C$. From this discussion and Proposition~\ref{propBH} (see also~\cite[\S 6]{favredujardin}) we get
\begin{proposition}
For any integer $0\le i \le d-2$,  $T_i$ extends to a positive closed $(1,1)$-current on  $\p^{d-1}_\C$ of mass $1$. Its support $\Gamma_i$ is the closure in $\p^{d-1}_\C$ of $\partial\{g_{c,a}(c_i)=0\}\subset \A_\C^{d-1}$, and  
\begin{center}
$\Gamma_i \cap H_\infty = \{[c:a:0]\in\p^{d-1}_\C, \,  P_{c,a}(c_i)=0\}~.$
\end{center}
\label{propDF}
\end{proposition}

Observe that $T_i$ does not have locally bounded potentials at points on $\Gamma_i\cap H_\infty$, and indeed admits positive Lelong numbers there.
However it admits continuous potential elsewhere and we may thus consider  any intersection product of the form $T_{i_1}\wedge \ldots \wedge T_{i_j}$ in the sense of pluripotential theory, see e.g.~\cite{BedfordTaylor}.

\begin{proposition}
For any integer $0\le i \le d-1$, we have $T_i\wedge T_i = 0$, and $\mu_\bif = T_0 \wedge\ldots\wedge T_{d-2} = T_\bif^{d-1}$ where $T_\bif = \frac{1}{d-1}\big(T_0+\cdots+T_{d-2}\big)$.
\end{proposition}
We refer to~\cite[\S 6 and \S 7]{favredujardin} for a proof. DeMarco has proved that the support of the current $T_\bif$ coincides with the bifurcation locus of the family $(P_{c,a})_{(c,a)\in\C^{d-1}}$ in the classical sense of Ma\~n\'e-Sad-Sullivan (see \cite{DeMarco1}).


\subsection{A good metric on the line bundle $\mathcal{O}(1)$.}
\label{sec:semipos-arch}

Let $L \to X$ be any holomorphic line bundle on a complex manifold $X$. A hermitian metric $h$ on $L$
is a way to assign to each local section $\sigma \in H^0(U,L)$ over an open subset $U\subset X$
a positive function $|\sigma|_h : U \to \R_+$ such that $|f\, \sigma|_h = |f| \, |\sigma|_h$ for any
$f\in \O(U)$. These functions are also supposed to be compatible under restrictions. 
In a local chart one may write $|\sigma|_h = |\sigma (z)| e^{-g(z)}$ for some real-valued function $g$. 
The metric is said to be \emph{continuous} when $g$ is continuous, and \emph{semi-positive} when $g$ is psh, i.e. 
when the curvature form $dd^c g$ of $h$ is a positive closed $(1,1)$ current. 
More generally  one can consider semi-positive singular metrics. By definition this is an assigment as above 
given locally in an open set $U$ by $|\sigma|_h = |\sigma (z)| e^{-g(z)}$ where $g: U \to [-\infty, + \infty)$ is an arbitrary psh function.

\medskip

The line bundle $\mathcal{O}(1)\to\p^{d-1}_\C$ is determined by the hyperplane at infinity $H_\infty$ so that any section $\sigma$ on an open subset $U$ of $\p^{d-1}_\C$ can be identified with a meromorphic map $\sigma : U \to \C$ that has poles of order $\le 1$ along $H_\infty$. Using this trivialization a continuous hermitian metric  on $\mathcal{O}(1)$ is given by a non-negative function $g : \p^{d-1}_\C \to \R \cup \{ + \infty\}$ such that $g - \log | z |$ is continuous on any open set where $H_\infty = \{ z=0\}$. We shall denote by $|\sigma|_g := |\sigma|e^{-g}$ the associated metric. The metric is semipositive when $g$ is psh on $\mathbb{A}^{d-1}_\C$.

\begin{proposition}
The metric $|\cdot|_G$ is a continuous semi-positive metric on $\mathcal{O}(1)$.
\label{propmetric}
\end{proposition}
Our proof is similar to the arguments given in~\cite[\S~7]{GHT}.
\begin{proof}
One needs to show that 
\begin{eqnarray*}
\tilde{G}:= G - \log^+\max\left\{|c|,|a|\right\}
\end{eqnarray*}
extends to a continuous function on $\p^{d-1}_\C$. It follows from Proposition \ref{propBH}  that $\tilde{G}$ is bounded  near infinity. Recall that $g_k (c,a):= g_{c,a}(c_k)$, $T_k$ is the extension of mass $1$ of $dd^c g_k$ to $\p^{d-1}_\C$, and $\Gamma_k$ denote its support. Observe that for any $k$, the semi-positive singular hermitian metric $|\cdot|_{g_k}$ is continuous on $\p^{d-1}_\C\setminus\Gamma_k$ since its curvature is zero outside $\Gamma_k$.

 Pick any point $x\in H_\infty$ and choose coordinates near $x$ such that $H_\infty = \{ z=0\}$. Pick any $0\le i \le d-2$, and suppose first that $x \notin \Gamma_i$. Then $g_i -\log|z|$ is pluriharmonic near $x$, hence $g_i \ge \log|z| - A$ for some constant $A$. Suppose on the other hand that  $x \in \Gamma_i$. In a neighborhood of $x$, we have $|P_{c,a}(c_i)| \le \varepsilon\, \max \{ |c_k|, |a| \}^d$ by  Proposition~\ref{propBH} and~\ref{propDF}. Here $\varepsilon$ is a positive constant that can be chosen arbitrarily small. By \cite[Lemma 6.4]{favredujardin}, it follows that
\begin{center}
$g_i(c,a) =\displaystyle \frac1d g_{c,a}(P_{c,a}(c_i)) \le \log \max \{ |c_k|, |a| \} + \frac{\log C}{(d-1)d} + \frac{\log \varepsilon}{d} $
\end{center}
Shrinking the neighborhood if necessary, we may thus assume that 
\begin{center}
$\displaystyle G(c,a) = \max_{x \notin \Gamma_i}  g_i (c,a)$
\end{center}
near $x$. It now follows easily that $G - \log|z|$ is continuous near $x$.
\end{proof}


\section{At a non-archimedean place.}\label{sec:nonarchi}
We extend the results of the previous section to a non-archimedean metrized field $(\KK, |\cdot|_v)$. 
We construct the local Green function $g_{c,a,v}$ in Proposition~\ref{prop:localgreen}, and give precise estimates on
$G_v(c,a) = \max_i g_{c,a}(c_i)$ (Proposition~\ref{BHnonarch}) that imply that the line bundle $\O(1)\to\p^{d-1}_\KK$ can be endowed with a semipositive metric
in the sense of Zhang (see \S\ref{sec:semipos-nonarch}).


\subsection{Local Green functions.}

Let $\mathbb{K}$ be an algebraically closed field of characteristic zero equipped with a \emph{non-archimedean} absolute value $|\cdot|_v$. 

For any $(c,a) \in \KK^{d-1}$, we may consider the polynomial $P_{c,a}$ acting on $\KK$. We first list (classical) estimates that will be important in the sequel.

Again we write $|c|_v \pe \max_{1\le i \le d-2} |c_i|_v$.
\begin{lemma}\label{lupperbound}
There exists a constant $\alpha_v\geq1$ such that
\begin{equation}\label{eq:upperboundnonarch}
|P_{c,a}(z)|_v\leq \alpha_v\, \max\left\{|c|_v,|a|_v, |z|_v \right\}^d~.
\end{equation}
for any $(c,a)\in \KK^{d-1}$ and any $z\in\KK$. When the residual characteristic of $\KK$ is greater than $d+1$, then we may take $\alpha_v =1$. 
\end{lemma}

\begin{proof}
This follows immediately from~\eqref{eq:defpoly}, and the non-archimedean triangle
inequality.  Observe first that
\begin{equation}\label{eq:symnonarch}
|\sigma_{j}(c)|_v \le \max_{1\le k\le d-2} \{ |c_k|_v\}^j = |c|_v^j~.
\end{equation}
Whence
\begin{center}
$|P_{c,a}(z)|_v
\le \max \{ |d|_v^{-1}  \, |z|_v^d, |j|_v^{-1}  \, |c|_v^{d-j} \, |z|_v^j , |a|_v^d\}
\le  \max_{j\le d} \{ |j|_v^{-1}\} \, \max\{ |c|_v, |a|_v, |z|_v\}^d$
\end{center}
as required, with $\alpha_v := \max_{j\le d} \{ |j|_v^{-1}\} $.
\end{proof}
\begin{lemma}\label{llowerbound}
Write
\begin{center}
$C_v(c,a) := \max \left\{ |d|_v^{1/(d-1)}, |d|_v^{1/d}\, |a|_v, \max_{2\le j \le d-1} |\sigma_{d-j}(c)|_v^{1/(d-j)} |d/j|_v^{1/(d-j)}\right\}
~.$
\end{center}
Then for any $z \in \KK$ such that $|z|_v \ge C_v(c,a)$, we have
\begin{equation}\label{eq:lowerboundnonarch}
|P_{c,a}(z)|_v = |d|_v^{-1}\, |z|_v^d \ge |z|_v~.
\end{equation}
\end{lemma}
\begin{proof}
If $|z|_v\ge C_v(c,a)$, we have
\begin{center}
$\displaystyle\max_{2\leq j\leq d-1}\left\{\left|\frac{\sigma_{d-j}(c)}{j}z^j\right|_v\, ,\, |a|^d_v\right\}\leq\left|\frac{1}{d}z^d\right|_v=|d|_v^{-1}\, |z|_v^d$
\end{center}
and the non-archimedean triangle inequality gives $|P_{c,a}(z)|_v = |d|_v^{-1}\, |z|_v^d$. Since $|z|_v\ge |d|_v^{1/(d-1)}$, we have $|d|_v^{-1}|z|_v^{d-1}\ge 1$, which ends the proof.
\end{proof}
The previous two estimates imply
\begin{proposition}
Write
\begin{center}
$ \tC_v(c,a) := \max \{ |a|_v, |c|_v, C_v(c,a)\}~,$ 
\end{center}
and set $h_{c,a,v} (z) := \log \max \{ \tC_v(c,a), |z|_v\}$. Then 
we have
\begin{eqnarray}
d^{-1}  h_{c,a,v} \circ P_{c,a}(z) & \ge&  h_{c,a,v} (z)  + \min \left\{ \frac1d \log |d|_v^{-1}, \left(\frac{1}{d}-1\right) \log \tC_v(c,a) \right\} \label{eqdown}
\\
d^{-1}  h_{c,a,v} \circ P_{c,a}(z) &\le&  h_{c,a,v} (z)  +  \frac1d \log \alpha_v
\label{equp}
\end{eqnarray}
\end{proposition}
\begin{proof}
Suppose $|z|_v \ge \tC_v(c,a)$. Then~\eqref{eq:lowerboundnonarch} implies
$$|P(z)|_v = |d|_v^{-1} |z|_v^d\ge |z|_v \ge \tC_v(c,a)$$ hence 
\begin{center}
$d^{-1}  h_{c,a,v} \circ P_{c,a}(z) = d^{-1} \log |P_{c,a}(z)|_v = \log|z|_v + d^{-1} \log |d|_v^{-1}$.
\end{center}
When  $|z|_v \le \tC_v(c,a)$, then $h_{c,a,v}(z) = \log \tC_v(c,a)$ and
\begin{center}
$d^{-1}  h_{c,a,v} \circ P_{c,a}(z) \ge  d^{-1} \log \tC_v(c,a)  = h_{c,a,v}(z) + \left(\frac1d-1\right) \log \tC_v(c,a)$.
\end{center}
These two inequalities imply~\eqref{eqdown}.

For the upper bound, suppose again $|z|_v \ge \tC_v(c,a)$. Then~\eqref{eq:upperboundnonarch} implies
$|P(z)|_v \le \alpha_v |z|_v^d$ hence
$$d^{-1}  \log  | P_{c,a}(z) |_v
\le \log|z|_v + \frac1d \log \alpha_v  = h_{c,a,v}(z) + \frac1d \log \alpha_v.
$$
If $| P_{c,a}(z) |_v\ge \tC_v(c,a)$ then~\eqref{equp} is clear. If  
$| P_{c,a}(z) |_v\le \tC_v(c,a)$, then
\begin{center}
$d^{-1} h_{c,a,v} \circ P_{a,c}(z) = d^{-1}\log \tC_v(c,a) \le \log \tC_v(c,a) \le \log|z|_v = h_{c,a,v}(z)$.
\end{center}
The last case is when   $|z|_v \le \tC_v(c,a)$, so that by~\eqref{eq:upperboundnonarch} we have
\begin{center}
$d^{-1} h_{c,a,v} \circ P_{a,c}(z) \le \max \{ d^{-1} \log \tC_v(c,a), d^{-1}\log \alpha_v + \log \tC_v(c,a) \}$.
\end{center}
This concludes the proof.
\end{proof}
All these estimates imply the following key
\begin{proposition}\label{prop:localgreen}
For any constant $C >0$, the sequence of functions $d^{-n} h_{c,a,v} \circ P_{c,a}^n(z)$
converges uniformly on sets of the form
$\{ (z,a,c) \in \KK^d, \, \max \{ |a|_v, |c|_v \} \le  C \}$.

The function
\begin{center}
$g_{c,a,v}(z) := \lim_{n\to \infty} d^{-n} h_{c,a,v} \circ P_{c,a}^n(z)$
\end{center}
thus defines a continuous non-negative function on $\KK^d$
that satisfies $g_{c,a,v} \circ P_{c,a} = d g_{c,a,v}$ and 
$\{z \in\KK, \,  g_{a,c,v}(z) =0 \} = \{ z \in \KK, \, |P^n_{c,a}(z)|_v = O(1)\}$.
\end{proposition}
For the record we also observe that~\eqref{equp} and an immediate induction implies
\begin{equation} \label{eq:gnottoobig}
g_{c,a,v}(z) \le h_{c,a,v}(z) + \frac{1}{d-1} \log \alpha_v
\end{equation}
The function $g_{c,a,v}$ is called the $v$-\emph{adic Green function} of $P_{c,a}$. 


\subsection{Green function on the parameter space.}
As in the archimedean case, we define $g_{j,v}(c,a)\pe g_{c,a,v}(c_j)$ for $0\leq j\leq d-2$ and
\begin{center}
$G_v(c,a)\pe \max _{0\le i \le d-2} g_{i,v}(c,a)$.
\end{center}
Our aim is to prove an analog of Proposition~\ref{propmetric} in a non-archimedean context.

\medskip

Recall that the set  $\p^{d-1}(\KK)$ of $\KK$-points of the projective space can be endowed with the following projective metric: 
\begin{eqnarray*}
d_{\p^{d-1}(\KK)}
(  [x_0: \ldots : x_{d-2}] , [x'_0: \ldots : x'_{d-2}] ) := 
\frac{ \max_i \{ |x_i x'_0 - x'_i x_0|_v \}  }{\max_i |x_i|_v\, \max_i |x'_i|_v}
\end{eqnarray*}
The rest of this section is devoted to the proof of the following result.
\begin{proposition}\label{BHnonarch}
For each $n\in \mathbb{N}$,  the function 
\begin{equation}\label{eq:semipos}
H_n (c,a) := \max_{0\le i \le d-2}  \left\{ \frac1{d^n} \log^+| P^n_{c,a}(c_i)|_v \right\}   - \log^+ \max \{|c|_v, |a|_v\} \end{equation}
extends to a continuous function on $\p^{d-1}(\KK)$, and the sequence $H_n$ converges uniformly to $G_v- \log^+ \max \{ |c|_v, |a|_v\}$ on $\KK^{d-1}$.
Moreover,  we have $G_v (c,a) =  \log^+ \max \{ |c|_v, |a|_v\}$  when the residual characteristic of $\KK$ is larger than $d+1$.
\end{proposition}

\begin{proof}
Suppose first the residual characteristic of $\KK$ is larger than $d+1$. Then
$\alpha_v = 1$, and $C_v(c,a) \le \max \{|c|_v,|a|_v\}$ hence $\tC_v(c,a) = \max \{|c|_v,|a|_v\}$.
Suppose  that $\max \{ |c|_v,|a|_v\} \le 1$. Then by induction~\eqref{eq:upperboundnonarch} implies
$|P^n_{c,a}(c_i)|_v \le 1$ for all $n$ and for all $0\le i\le d-2$, hence $G_v(c,a) =0$.

Conversely assume  $\max \{|c|_v,|a|_v\} \ge 1$.
Suppose  $|c_i|_v = \max \{|c|_v,|a|_v\}$ for some $1\le i\le d-2$ (the case $|a|_v \ge |c|_v$ can be treated analogously). Then $|c_i|_v \ge C_v(c,a)$ and~\eqref{eq:lowerboundnonarch} implies by induction that  $|P^n_{c,a}(c_i)|_v = |c_i|_v^{d^n}$, hence  $g_{c,a,v}(c_i) = \log |c_i|_v$. For any $ j \neq i$, we also have $|P^n_{c,a}(c_j)|_v \le |c_i|_v^{d^n}$ by~\eqref{eq:upperboundnonarch} whence  $G_v(c,a) =  g_{c,a,v}(c_i) =  \log |c_i|_v = \log^+ \max \{ |c|_v, |a|_v\}$ as required.

\smallskip
 
To prove the other statements, we shall need the following lemmas.

\begin{lemma}\label{ltropfort}
There exists a constant $C_v>0$ such that for all $0\le i\le d-2$, and for all $C\ge1$ and $\varepsilon >0$
such that $  C^d\varepsilon > \max \{ 1, C_v C\}$, then
we have 
\begin{equation}\label{eqcoool1}
\frac1{d^{n+1}} \log |P^{n+1}_{c,a}(c_i)|_v = \frac1d \log |P_{c,a}(c_i)|_v -  (\sum_0^{n-1} d^{-l}) \, \log |d|_v
\end{equation}
on the open set 
$$
U_i( \varepsilon, C):= \left\{ (c,a) \in \KK^{d-1}, \, \max \{ |c|_v, |a|_v\} \ge C, \, |P_{c,a}(c_i)|_v \ge \varepsilon \max \{|c|_v,|a|_v\} ^d\right\}~.
$$
In particular, 
\begin{equation}\label{eqcoool2}
g_{c,a,v}(c_i) = \frac1d \log |P_{c,a}(c_i)|_v - \frac1{d-1} \, \log |d|_v >0
\end{equation}
and $\frac1{d^{n}} \log |P^{n}_{c,a}(c_i)|_v \to g_{c,a,v}(c_i)$ uniformly on $U_i( \varepsilon, C)$.
\end{lemma}

\begin{lemma}\label{lup}
Pick $C\geq1$ and $\varepsilon >0$ such that $ C^d\varepsilon > \max \{ 1, C_v C\}$ as above. Then
for any $0\le i \le d-2$, 
we have 
\begin{equation}\label{eqeasyup}
g_{c,a,v}(c_i) \le  \log \max \{|c|_v,|a|_v\}  +\frac1d \log \e  + \frac{\log \alpha_v}{d(d-1)}
\end{equation}
for any $(c,a) \notin U_i(\varepsilon, C)$ with $\max \{ |c|_v, |a|_v \} \ge C$ .
\end{lemma}

\begin{lemma}\label{llastcall}
Pick $C\geq1$ and $\varepsilon >0$ such that $ C^d\varepsilon > \max \{ 1, C_v C\}$ as above.
For any two distinct indices $0\le i,j\le d-2$, then
$\max \{ |P^{n}_{c,a}(c_i)|_v,|P^{n}_{c,a}(c_j)|_v\}  = |P^{n}_{c,a}(c_i)|_v$ for all $n\ge 1$, and 
all $(c,a) \in U_i(\alpha_v^d\,\e,C) \setminus U_j(\e,C)$. Moreover
$$
\frac1{d^{n}} \log  \max \{ |P^{n}_{c,a}(c_i)|_v,|P^{n}_{c,a}(c_j)|_v\}  \to \max\{ g_{c,a,v}(c_i) , g_{c,a,v}(c_j)\}
$$
uniformly on $U_i(\alpha_v^d\,\e,C) \setminus U_j(\e,C)$.
\end{lemma}

\begin{lemma}\label{lminegP}
There exists a constant $\beta_v\leq1$ such that for any $(c,a)\in \KK^{d-1}$, one has
$$
\max_{0\leq j\leq d-2}|P_{c,a}(c_j)|_v \ge
\beta_v\max\left\{|c|_v,|a|_v\right\}^d~.$$
In other words, we have $$\bigcup_{0 \le i \le d-2} U_i ( \varepsilon, C) = \left\{ (c,a) \in \KK^{d-1}, \, \max \{ |c|_v , |a|_v \}  \ge C \right\}~,$$
for any $C \ge 1$ and any $0<\varepsilon<\beta_v$.
\end{lemma}

We shall first  prove that $\max_{0\le i \le d-2}  \left\{ \frac1{d^n} \log^+| P^n_{c,a}(c_i)|_v \right\}$ converges uniformly to $G_v$. Pick $C\gg 1$ and $\varepsilon >0$ such that $\beta_v > \varepsilon > C_v C^{1-d}$, and  $C^d\varepsilon > 1$.

\smallskip

On the set $B:= \{ (c,a), \, \max \{ |c|_v, |a|_v \} \le C\}$ then Proposition~\ref{prop:localgreen} implies that
$h_n:= \frac1{d^n} \log \max \{\tC_v(c,a), |P^n_{c,a}(c_i)|_v \}$ converges uniformly to $g_{c,a}(c_i)$ for all $0\le i \le d-2$. Now observe that $0 \le \tC_v(c,a) \le CC_v$ is uniformly bounded on $B$, whence $$\sup_B \left| h_n - \frac1{d^n} \log^+ |P^n_{c,a}(c_i)|_v \right|\le \frac1{d^n} \log (C C_v) \to  0~.$$
It follows that $\frac1{d^n} \log^+ \max_i \{ |P^n_{c,a}(c_i)|_v\} $ converges uniformly to $G_v$ on $B$. 
 
\smallskip

Lemma~\ref{lminegP} implies that the complement of $B$ is covered by the open sets
$U_{I,J}:= \cup_I U_i(\alpha_v^d\e,C) \setminus \cup_J U_j(\e,C)$ where $I,J$ range over all subsets of $\{ 0, \ldots , d-2\}$ such that $I\cap J = \emptyset$ and $I \cup J = \{ 0, \ldots , d-2\}$. 
Lemma~\ref{llastcall} shows that
$$
\max_{0\le i \le d-2} \frac1{d^n}  \log |P^n_{c,a}(c_i)|_v  = \max_{i \in I } \frac1{d^n}  \log |P^n_{c,a}(c_i)|_v  \to G_v
$$
uniformly on $U_{I,J}$.  This proves~\eqref{eq:semipos} since in view of~\eqref{eqcoool2},
$g_{c,a,v}(c_i)|_{U_{I,J}} >0$ for any $i \in I$, hence 
$\max_{i \in I } \frac1{d^n}  \log |P^n_{c,a}(c_i)|_v = \max_{i \in I } \frac1{d^n}  \log^+ |P^n_{c,a}(c_i)|_v$ for $n$ large enough.

\smallskip

We next prove that 
$$H_n:= \frac1{d^n}  \max _{0\le i \le d-2} \log^+ |P^n_{c,a}(c_i)|_v - \log^+ \max \{ |c|_v, |a|_v\}$$
extends continuously to $\p^{d-1}(\KK)$.
Since all polynomials $P^n_{c,a}(c_i)$ are homogeneous of degree $d^n$, 
the function  $$(c,a) \mapsto \frac{\max _{0\le i \le d-2} |P^n_{c,a}(c_i)|_v}{\max \{ |c|_v, |a|_v\}^{d^n}}$$
 is well-defined on $\p^{d-1}(\KK)$ and continuous. It follows that to prove that 
$H_n$ extends continuously to $\p^{d-1}(\KK)$, it is sufficient to check that it is bounded from below near $H_\infty$.
On any open subset $U_{I,J}$ as above, Lemma~\ref{llastcall} and~\eqref{eqcoool1} imply
\begin{multline*}
H_{n+1} = \frac1{d^{n+1}}  \max _{i \in I } \log |P^{n+1}_{c,a}(c_i)|_v - \log \max \{ |c|_v, |a|_v\}
= \\ \frac1{d}  \max _{i \in I} \log |P_{c,a}(c_i)|_v - \log \max \{ |c|_v, |a|_v\} -\sum_{l=0}^{n-1} d^{-l} \log|d|_v \ge \frac1d \log \e 
\end{multline*}
as required, since $|d|_v\leq 1$.

We have thus proved that $H_n$ is a sequence of continuous functions on $\p^{d-1}(\KK)$ that converges uniformly to $G_v - \log^+ \max \{ |c|_v, |a|_v\}$, hence the latter function is continuous.
This concludes the proof of Proposition~\ref{BHnonarch}.
\end{proof}

\begin{proof}[Proof of Lemma~\ref{ltropfort}]
We begin  observing that
$$
C_v(c,a) \le C_v \, \max \{ |c|_v, |a|_v \}
$$
for some constant $C_v>0$ depending only on $d$ and $v$.

 For any $(c,a) \in U:= U_i( \varepsilon, C)$,  we have
$$C_v(c,a) \le C_v \, \max \{ |a|_v, |c|_v\} \le C_v C^{1-d} \max \{ |a|_v, |c|_v\} ^d
\le C_v C^{1-d} \varepsilon^{-1} |P_{c,a}(c_i)|_v~.
$$
 Whence~\eqref{eq:lowerboundnonarch} implies 
\begin{center}
$|P^n_{c,a}( P_{c,a}(c_i) )|_v = |d|_v^{ -1 - \ldots - d^{n-1}} |P_{c,a}(c_i)|_v^{d^n}$
\end{center}
for all $n$ if $ C_v C^{1-d} \varepsilon^{-1} \le 1$, in which case
$g_{c,a,v}(c_i) = \frac1d \log |P_{c,a}(c_i)|_v - \frac1{d-1} \, \log |d|_v$
follows from the functional equation $g_{c,a} (c_i) = d^{-1} g_{c,a} (P_{c,a} (c_i))$.
By assumption $C^d\e >1$, hence $|d|_v\leq 1$ implies
$g_{c,a,v}(c_i) \ge \frac1d \log |P_{c,a}(c_i)|_v \ge \log (\e C^d) >0$.
\end{proof}

\begin{proof}[Proof of Lemma~\ref{lup}]
By~\eqref{eq:gnottoobig}, we get
\begin{center}
$g_{c,a,v}(c_i) = \frac1{d} g_{c,a,v}(P_{c,a}(c_i)) \le \frac1{d} \log \max \{ |P_{c,a}(c_i)|_v, \tC_v(c,a)\} + \frac1{d(d-1)} \log \alpha_v~.$
\end{center}
Now replacing $C_v$ by greater constant if necessary, we have $\tC_v(c,a) \le C_v \max \{ |a|_v, |c|_v \}$,
so that 
\begin{eqnarray*}
g_{c,a,v}(c_i) & \le & \frac1{d} \log \max 
\left\{  \varepsilon \max \{ |a|_v, |c|_v\} ^d,  C_v  \max \{ |a|_v, |c|_v\}\right\} + \frac{\log \alpha_v}{d(d-1)} \\
& \le & \log \max \{ |a|_v, |c|_v\} + \max \left\{ \frac1d \log \e ,  ( 1 - \frac1d) \log C + \frac1d \log C_v \right\} 
+ \frac{\log \alpha_v}{d(d-1)} 
\end{eqnarray*}
since $(c,a) \notin U_i(\e,C)$,  $\max \{ |a|_v, |c|_v \}\ge 1$ and $C\geq1$.
By assumption we have $\log \e \ge (1-d) \log C + \log C_v$ whence
$$
g_{c,a,v}(c_i) \le
\log \max \{ |a|_v, |c|_v\} + \frac1d \log \e  + \frac{\log \alpha_v}{d(d-1)}~.$$
This concludes the proof.
\end{proof}

\begin{proof}[Proof of Lemma~\ref{llastcall}]
On $U_i(\alpha_v\e,C)\setminus U_j(\e,C)$ we have
$$
|P^n_{c,a}(c_i)|_v = |P_{c,a}(c_i)|_v^{d^{n-1}} \, |d|_v^{-1 - \ldots - d^{n-2}}
\ge (\alpha_v^d\e)^{d^{n-1}} \max \{ |c|_v, |a|_v\}^{d^{n}} \, |d|_v^{-1 - \cdots - d^{n-2}}
$$
by~\eqref{eqcoool1}
and
$$
|P^n_{c,a}(c_j)|_v
\le \alpha_v^{1+ \ldots + d^{n-1}} \e^{d^{n-1}}\, \max \{ |c|_v, |a|_v\}^{d^n}
$$
by iterating~\eqref{eq:upperboundnonarch} and using $\max\{|c|_v, |a|_v\}  \ge C$ hence 
$\e \max\{|c|_v, |a|_v\}^d \ge \max\{|c|_v, |a|_v\}$.
It follows that 
$$
\frac{|P^n_{c,a}(c_i)|_v}{|P^n_{c,a}(c_j)|_v}
\ge |d|_v^{-1 - \ldots - d^{n-2}} \alpha_v^{d^n - \frac{d^n-1}{d-1}} \ge 1~.
$$
The uniform convergence then follows from Lemma~\ref{ltropfort}.
\end{proof}

\begin{proof}[Proof of Lemma~\ref{lminegP}]
Let $\mathcal{I}\subset \KK[c_1, \ldots, c_{d-2}, a]$ be the ideal generated by the homogeneous polynomials $\{P_{c,a}(c_i)\}_{0\leq i\leq d-2}$. By Lemma~\ref{lminfinity}  these generators have no  common zero other than $(0,\ldots,0)$, hence  $\sqrt{\mathcal{I}}=(c_1,\ldots,c_{d-2},a)$ by the Hilbert's Nullstellensatz. 

In particular, there exists $m\geq d$ and for $1\leq i\leq d-2$, homogeneous polynomials $Q_{i,j}=\sum_{|I|=m-d}q_{i,j,I}X_1^{i_1}\cdots X_{d-1}^{i_{d-1}}\in\mathbb{Q}[X_1,\ldots,X_{d-1}]$ of degree $m-d$ such that
$$c_i^m=\displaystyle\sum_{j=0}^{d-2}Q_{i,j}(c,a)P_{c,a}(c_j)~.$$
We thus have
\begin{eqnarray}
|c_i|_v^m & \leq & C_Q\, \cdot\max\left\{|c|_v^{m-d},|a|_v^{m-d}\right\}\max_{0\leq j\leq d-2}|P_{c,a}(c_j)|_v,
\label{inegck}
\end{eqnarray}
 for any $1\leq i\leq d-2$, where $C_Q = \max_{0\leq i,j\leq d-2, \atop |I|=m-d}\left|q_{i,j,I}\right|_v$.
Observe now that $a^d=P_{c,a}(c_0)$, hence
\ref{inegck} gives
$$
\max\left\{|c|_v^d,|a|_v^d\right\}\leq \beta_v^{-1}\max_{0\leq j\leq d-2}|P_{c,a}(c_j)|_v,$$
with $\beta_v^{-1}:= \max \{ 1, C_Q\}$.
\end{proof}


\subsection{Semi-positive continuous non-archimedean metric.}\label{sec:semipos-nonarch}

We refer to~\cite{ACL2} for the material of this section. We assume that $(\KK,|\cdot|_v)$ is a 
non-archimedean local field of characteristic zero, i.e. a finite extension of $\Q_p$ and let $\KK^0$ be its ring of integers.

Let us first recall how to define the $\KK$-analytic space in the sense of Berkovich associated with  $\p^{d-1}_\KK$. The projective space is obtained by patching together $d$ copies of the affine space, and in a similar way
its $\KK$-analytic avatar $\p^{d-1,\an}_\KK$ is obtained  by patching together $d$ copies of the space of multiplicative semi-norms on $\KK [x_1, \ldots , x_{d-1}]$ extending the norm on $\KK$ endowed with the topology of the pointwise convergence. The space $\p^{d-1, \an}_\KK$ is compact and Hausdorff. Itis naturally endowed with a structural sheaf of analytic functions $\O_{\p^{d-1}}^{\an}$. 

\smallskip

A line bundle on  $\p^{d-1, \an}_\KK$ is a rank $1$ invertible coherent sheaf. 
For any hyperplane $H$, and any $k\in \Z$ the sheaf of local meromorphic functions whose
divisor is $\ge - k H$ defines a line bundle $\O(k)$.  Global sections of $\O(k)$ are rational functions on 
 $\p^{d-1}_\KK$ whose divisor is $\ge - k H$, and the vector space of all sections of $\mathcal{O}(k)$
can be identified with the space of 
homogeneous polynomials of degree $k$. Any line bundle on a projective space
is isomorphic to some $\O(k)$.

\smallskip

By definition a (\emph{continuous}) \emph{metric} on  $\O(k) \to \p^{d-1, \an}_\KK$ is the data for each local continuous section $s$ of $\O(k)$ defined on an open set $U$ of a  continuous function $\|s\| : U \to \R_+$ such that $\| fs \| = | f|_v \, \|s\|$ for all analytic function $f$, and  $\|s\| (x) = 0$ iff $s(x) =0$. We also impose
natural compatibility conditions for these functions with respect to restrictions. 

Just as in the archimedean case, any non-negative function $g : \p^{d-1}_\C \to \R \cup \{ + \infty\}$ such that $g - \log | z |_v$ is continuous on any open set where $H_\infty = \{ z=0\}$ induces
a metric $|\cdot |_g$ on $\O(1)$ such that  for any degree $1$ homogeneous polynomials $Q$ viewed as a global section of $\O(1)$ as above one has $| Q|_g = |Q|_v \, e^{- g}$.

\smallskip

A \emph{model} of the line bundle $\O(k) \to \p^{d-1, \an}_\KK$ over $\KK^0$ is the choice of 
\begin{itemize}
\item
a  normal $\KK^0$-scheme $\mathfrak{X}$ that is projective and flat over $\spec \KK^0$, and has generic fiber isomorphic to $\p^{d-1, \an}_\KK$;
\item
 a hypersurface $\mathfrak{H}$ of $\mathfrak{X}$ whose generic fiber is equal to $kH$.
\end{itemize}
Any such model determines a metric on $\O(k)$ as follows.  We cover $\mathfrak{X}$ by affine charts $\mathfrak{U}_i = \spec B_i$ for some finitely generated $\KK^0$-algebras $B_i$ such that the set $A_i\subset \p^{d-1, \an}_\KK$ of bounded semi-norms on $B_i \otimes_{\KK^0} \KK$ forms a (closed) cover of 
$\p^{d-1, \an}_\KK$. We also choose $h_i \in B_i$ determining $\mathfrak{H}$ on $\mathfrak{U}_i$. 
Observe that for any other choice $h'_i$ we have $| h_i / h'_i (x) |_v = 1$ on $A_i$. It follows that we may define
in a unique way a continuous metric by setting $|\sigma|_\mathfrak{H} (x) \pe | \sigma h_i (x)|_v$ for any  section $\sigma$ of $\O(k)$ and for any $x \in A_i$ where $\sigma h_i$ is viewed as an element of $B_i \otimes_{\KK^0} \KK$.

A \emph{semipositive model metric} $|\cdot |$ on $\O(k)$ is by definition a metric such that we can find an integer $e\ge 1$ and a hypersurface $\mathfrak{H}$ that is \emph{nef} over $\spec \KK^0$
and whose generic fiber has degree $ke$ and 
satisfying  $|\sigma | = | \sigma^e |^{1/e}_{\mathfrak{H}}$ for any local section $\sigma$ of $\O(k)$.
By definition a \emph{semipositive metric} is a uniform limit of  semipositive model metrics.

\medskip

The projective space $\p^{d-1}_\KK$ admits a canonical model $\p^{d-1}_{\KK^0}$ over $\KK^0$
with affine charts of the form $$\spec \KK^0 \left[\frac{z_0}{z_j}, \ldots ,  \frac{z_{j-1}}{z_j}, \frac{z_{j+1}}{z_j}, \ldots \frac{z_{d-1}}{z_j}\right]~,$$
where $[z_0 : \ldots : z_{d-1}]$ are homogeneous coordinates. 
For $k\ge 0$ the hypersurface $[z_0 = 0]$ is nef over $\spec \KK^0$, and 
the induced semipositive model metric $|\cdot|_\nv$ on $\O(1)$ satisfies 
\begin{equation}\label{eqnv}
|Q|_\nv := \frac{|Q|_v}{\max \{ |z_0|_v, \ldots , |z_{d-1}|_v\}}~, 
\end{equation}
for any homogeneous polynomial of degree $1$. Observe that with the above notation
$|\cdot |_\nv = |\cdot |_g$ with $g = \log^+ \max \left\{ \left| x_1\right|_v, \ldots , \left| x_{d-1} \right|_v\right\}$ where $x_i = \frac{z_i}{z_0}$.
 
More generally, for any choice of  $d$ homogeneous polynomials 
$\tilde{P}_0, \ldots , \tilde{P}_{d-1}$  of degree $q\ge 1$ such that $\cap_{0\le i \le d-1} \tilde{P}_i^{-1}(0) = (0)$ we get a natural metric on 
$\O(1)$ by setting 
$$
|Q|_{\tilde{P}} := \frac{|Q|_v}{\max \{ |\tilde{P}_0|_v^{1/q}, \ldots , |\tilde{P}_{d-1}|_v^{1/q}\}}~. 
$$
Here we have $|\cdot |_{\tilde{P}} = |\cdot |_g$ with $g = \frac{1}{q}\log\max_{0\le i \le d-1} \left\{ |\tilde{P}_i|_v (1, x_1, \ldots , x_{d-1}) \right\}$.
\begin{lemma}\label{lmlast}
The metric $|\cdot|_{\tilde{P}}$ is a semipositive model metric. 
\end{lemma}
\begin{proof}
Choose homogeneous coordinates and set 
$F[z] = F[z_0 : \ldots : z_{d-1}] \pe [\tilde{P}_0(z) :  \ldots : \tilde{P}_{d-1}(z)]$.
This defines an endomorphism $F : \p^{d-1}_\KK \to \p^{d-1}_\KK$ of degree $q$.
We have a natural commutative diagram
$$
\xymatrix{
F^* \O(1) \simeq \O(q)
 \ar[d] \ar[r] 
&
\O(1) \ar[d]\\
\p^{d-1}_\KK
\ar[r]^F
&
\p^{d-1}_\KK
}
$$
and pulling-back  the metric $|\cdot |_\nv$ on $\O(1)$ by $F$ gives a metric $|\cdot|_\star$ on $\O(q)$ given by
\begin{equation}\label{eqstar}
|Q|_\star\pe \frac{|Q|_v}{\max \{ |\tilde{P}_0|_v, \ldots , |\tilde{P}_{d-1}|_v\}}
\end{equation}
for any homogeneous polynomials $Q$ of degree $q$.

Now choose any model $\mathfrak{X}$ of $\p^{d-1}_\KK$ such that 
the map $F$ induces a regular map $\mathfrak{F} : \mathfrak{X} \to \p^{d-1}_{\KK^0}$.
For instance one may take $\mathfrak{X}$ to be the normalization of the graph of the rational map $\p^{d-1}_{\KK^0} \dashrightarrow \p^{d-1}_{\KK^0}$ induced by $F$.
It follows that the metric $|\cdot|_\star$ is equal to the model metric associated with 
the pull-back by $\mathfrak{F}$ of the hyperplane $[z_0 =0]$ in $\p^{d-1}_{\KK^0}$.
Since the nefness property is preserved by pull-back it follows that $|\cdot|_\star$ is semipositive
which implies the result in view of~\eqref{eqnv} and~\eqref{eqstar}.
\end{proof}

We now come to the non-archimedean analog of Proposition~\ref{propmetric}.
\begin{proposition}\label{propmetricnon}
The metric $|\cdot|_{G_v}$ is a continuous semi-positive metric on $\O(1) \to \p^{d-1}_\KK$. Moreover, $|\cdot|_{G_v}=|\cdot|_\nv$ when the residual characteristic of $\KK$ is larger than $d+1$.
\end{proposition}

\begin{proof}
It follows from Proposition~\ref{BHnonarch} and the definition of a semi-positive metric that it is sufficient to
check that the metric $|\cdot |_{g_n} $ is a semipositive model metric on $\O(1)$ where
$g_n \pe  \max_{0\le i \le d-2}  \left\{ \frac1{d^n} \log^+| P^n_{c,a}(c_i)|_v \right\} $.
Observe that $P^n_{c,a}(c_i)$ are polynomials of degree $d^n$ that satisfy
$$ P^n_{c,a} (c_i) = \frac1{d^{1+\ldots + d^{n-1}}}\, P_{c,a}(c_i) ^{d^{n-1}}  + Q (c,a) ~.$$
with $\deg(Q) < d^n$.
We now pick homogeneous coordinates $[z_0: \ldots : z_{d-1}]$ such that
$[1 : c:a]$ is identified to $(c,a)$, i.e. $c_i = z_i/z_0$ for $1\le i \le d-2$, and $a= z_{d-1}/z_0$. 
Set  $\tilde{P}_{d-1}\pe z_0^{d^n}$, and
 $\tilde{P}_i (z_0, \ldots , z_{d-1})\pe z_0^{d^n} P^n_{c,a}(c_i)$  for $0\le i \le d-2$.

The $(d-1)$ polynomials $\tilde{P}_i$ are homogeneous of degree $d^n$ in the $z_i$'s.
Their common zeroes in $\p^{d-1}_\KK$  is the intersection of $\bigcap_{0\le i \le d-2} \{P_{c,a}(c_i) = 0\}$
with the hyperplane at infinity
 which is empty by Lemma~\ref{lminfinity}. In other words, 
$\cap_{0\le i \le d-1} \tilde{P}_i ^{-1}(0)= (0)$ in $\KK^{d}$ and
it follows from Lemma~\ref{lmlast}
that $|\cdot|_{\tilde{P}}$ is a semipositive model metric.

Since by definition  $|\cdot|_{\tilde{P}} = |\cdot |_{g_n}$ the proof is complete.
\end{proof}


\section{The bifurcation semipositive adelic metric and height function.}\label{sec:bif}

Let us briefly review the setting for Yuan's theorem.


\subsection{Semipositive adelic metrics.}

We let $K$ be any number field and denote by $M_K$ the set of its places, i.e. of its 
multiplicative norms  modulo equivalence. In each equivalence
class $v\in M_K$, we pick a norm $|\cdot|_v$ normalized in the usual way such that
the product formula $\prod_{M_K} |x|_v^{n_v} = 1$ holds for any $x\in K$, see~\cite{Silvermandiophantine}. Here $n_v\ge 1$ is the degree of the extension of complete normed fields $K_v/\Q_v$.

We let $K_v$ be the completion of $K$ with respect to any place $v \in M_K$, and write $\C_v$ for the completion of the algebraic closure of $K_v$.

Now pick any projective variety $X$ over $K$ and let $L \to X$ be any ample line bundle. To simplify notation for any $v\in M_K$ we write $X_v$ for the analytic variety induced by projective variety induced by $X$  over $K_v$. This analytic variety has to be understood in the sense of Berkovich when $v$ is a finite place. 
Similarly we denote by $L_v$ the line bundle induced by $L$ on $X_v$.

A \emph{semi-positive adelic metric} on $L \to X$ 
is the data for each place $v \in M_K$ of a semipositive continuous metric on 
the induced line bundle $L_v \to X_v$ in the sense of Sections~\ref{sec:semipos-arch} and~\ref{sec:semipos-nonarch}. These metrics are subject to the following conditions:
\begin{itemize}
\item
for any archimedean place such that $K_v = \C$, the metric $|\cdot|_v$ is invariant under conjugation; 
\item
there exists a model $\fL \to \fX$ of $L\to X$ over the ring of integers of $K$ such that for all finite places $v$ except for a finite number of exceptions, the metric $|\cdot|_v$ is induced by
the model $\fL_v \to \fX_v$ over the  ring of integers of $K_v$.
 \end{itemize}

Recall from Sections~\ref{sec:semipos-arch} and~\ref{sec:semipos-nonarch} that for any $v\in M_\Q$, the function $$G_v (c,a)= \max \{ g_{c,a,v}(c_i), \, 0\le i\le d-2\}$$ on $\C_v^{d-1}$ induces a natural metric $ |\cdot|_{G_v}$ on  $\O(1) \to \p^{d-1}_{\C_v}$.

The next result immediately follows from the definitions and Propositions~\ref{propmetric} and~\ref{propmetricnon}.

\begin{theorem}\label{thm:adelic}
The collection of metrics $\{ |\cdot|_{G_v} \}_{v\in M_\Q}$ induces a semipositive adelic metric on 
$\O(1)\to \p^{d-1}_\Q$.
\end{theorem}


\subsection{The bifurcation height function.}

Let $\bar{L}$ be any semipositive adelic metric on an ample line bundle $L\to X$
over a projective variety $X$ of dimension $d-1$ over a number field $K$.
Such a metrization  induces a \emph{height function} $h_{\bar{L}}$ on $X (\overline{K})$ by setting for any $x\in X(\overline{K})$,
\begin{center}
$h_{\bar{L}}(x)\pe\displaystyle\frac{1}{\deg(x)}\sum_{v\in M_\mathbb{K}}\sum_{z\in O(x)}-\log|\sigma(z)|_v$
\end{center}
for any section $\sigma$ of $L$ which does not vanish on $O(x)$, where $O(x)$ is the orbit of $x$ under the action of $\textup{Gal}(\bar{K}/K)$ and $\deg(x)$ is the order of $O(x)$.

\medskip

The most basic height that can be obtained from a semipositive adelic metric
on $\O(1) \to \p^{d-1}_\Q$ is usually referred to as the \emph{naive height}. This height is induced by the adelic semipositive metric $\{|\cdot|_{\nv,v}\}_{v\in M_\Q}$ described in Section \ref{sec:semipos-nonarch}. For any point $x \in \bar{\Q}^{d-1}$, we have the following expression:
$$
h_\nv (x)
=\displaystyle
\frac{1}{\deg(x)}\sum_{v\in M_\mathbb{Q}}\sum_{z\in O(x)} \log^+|z|_v
 \ge 0~,$$
with the convention $|z|_v = |(z_1, \ldots, z_d)|_v \pe \max\{ |z_i|_v\}$. 
\medskip

For any polynomial $P_{c,a}$ with $(c,a)\in K^{d-1}$ for some number field $K$, one can then define a height function
$h_{P_{c,a}}(x) :=  \lim_{n\to \infty} \frac1{d^n} h_\nv (P^n_{c,a}(x))$. It is also a non-negative function, and it satisfies the  invariance property $h_{P_{c,a}} \circ P_{c,a} = d\, h_{P_{c,a}}$. The Northcott property implies that $\{ h_{P_{c,a}} = 0 \}$ coincides with the set of points with finite orbit, or in other words with the set of all preperiodic points.
Moreover, the following formula holds
$$
h_{P_{c,a}} (x) =\displaystyle
\frac{1}{\deg(x)}\sum_{v\in M_K}\sum_{z\in O(x)} 
n_v g_{c,a,v} (z)~,
$$
for any $x\in \bar{\Q}$.
\begin{theorem}\label{thm:defhgt}
Let $h_\bif : \p^{d-1}(\bar{\Q}) \to \R$
be the height function induced by the semipositive adelic metric given by Theorem~\ref{thm:adelic}.
Then for any $(c,a) \in \bar{\Q}^{d-1}$, we have
\begin{equation}\label{eqdefhgt}
h_\bif (c,a) = 
\frac{1}{\deg(c,a)}\sum_{v\in M_\Q} \sum_{z\in O(c,a)} G_v(z)~,
\end{equation}
where $O(c,a)$ is the orbit of $(c,a)$ under the action of $\textup{Gal}(\bar{\Q}/\Q)$ and $\deg(c,a)$ is the order of $O(c,a)$.
In particular,
$h_\bif (c,a)\ge0$, $\sup_{\bar{\Q}^{d-1}} | h_\bif - h_\nv| < \infty$, $h_\bif \le h_\ingram \le (d-1)\, h_\bif$,  and $h_\bif (c,a) = 0$ iff $P_{c,a}$ is postcritically finite.
\end{theorem}
Here we let $h_\ingram (c,a) \pe \sum_0^{d-2} h_{P_{c,a}}(c_i)$ be the height function used by P.~Ingram~\cite{Ingram}.

\begin{remark}
The set of postcritically finite polynomials with postcritical set of cardinality bounded from above is defined by polynomial equations with rational coefficients. It is hence an algebraic subvariety defined over $\Q$. This set is known to be zero dimensional hence finite since it is included in $\{ G_{\C_v} = 0\}$ for any place $v\in M_\Q$ and the latter set  is bounded by Propositions~\ref{propBH} or~\ref{BHnonarch}. In particular when $P_{c,a}$ is  postcritically finite then $c,a \in \bar{\Q}^{d-1}$. We refer to the recent paper by A.~Levy~\cite{Levy} for an extension of this result to the positive characteristic case. 
As observed by P.~Ingram, the estimate $\sup_{\bar{\Q}^{d-1}} | h_\bif - h_\nv| < \infty$
and the Northcott property implies a stronger statement, see~\cite[Corollaries~2,3]{Ingram}.
\end{remark}

\begin{proof}
The equation~\eqref{eqdefhgt} follows from the definition by taking a section of $\O(1)$
that vanishes along $H_\infty$. Since $G_v\ge 0$ for all $v$, we also have $h_\bif \ge0$. 

The difference between $h_\bif$ and the standard height function is uniformly bounded
since $G_v - \log^+\max \{|c|, |a|\}$ is bounded for each place $v$ and equal to $0$ if $v$ is finite and large enough.

Since $G_v\ge0$ at all places, it follows that 
$G_v (c,a) = \max \{ g_{c,a} (c_i)\} \le \sum_i g_{c,a,v} (c_i) \le (d-1)\, G_v (c,a)$, 
whence $h_\bif \le h_\ingram \le (d-1)\, h_\bif$.

Suppose $h_\bif(c,a) =0$. Then $h_\ingram (c,a) =0$, hence
$h_{P_{c,a}}(c_i) = 0$ for all $i$. By Northcott's property, $c_i$ is preperiodic.
Conversely for any postcritically finite polynomial with $(c,a)\in \bar{\Q}^{d-1}$ 
and for any place the orbit of each critical point is bounded hence
$G_v(c,a) = \max \{ g_{c,a,v}(c_i) \} =0$, and $h_\bif(c,a) =0$.
\end{proof}


\subsection{Yuan's equidistribution Theorem.}
Let $\bar{L}$ be any semipositive adelic metric on an ample line bundle $L\to X$
over a projective variety $X$ of dimension $d-1$ over a number field $K$. Recall that all analytic spaces
$X_v$ are compact for any archimedean and non-archimedean places $v\in M_K$. 

It is possible to define for each $v$ a positive measure $c_1(\bar{L})_v^{d-1}$ on $X_v$. In the archimedean case, in a local trivialization 
where the metric can be written under the form $|\cdot| e^{-g}$ with $g$ psh and continuous
then $c_1(\bar{L})_v^{d-1}$ is equal to the Monge-Amp\`ere measure $(dd^c)^{d-1} g$. In the non-archimedean case, the construction is more involved and we refer to~\cite{ACL} for detail.

\smallskip

We say that a sequence of $0$-dimensional subvarieties $Z_m\subset X$ that are defined over $K$ (or equivalently finite sets that are invariant under $\textup{Gal} (\bar{K}/K)$) is \emph{generic} if for any divisor $D \subset X$ defined over $K$ then $Z_m\cap D = \emptyset$ for all $m$ large enough.
It is called \emph{small} if $h_{\bar{L}}(Z_m):= \frac1{\card (Z_m)}\sum_{x\in Z_m} h_{\bar{L}}(x)$
tends to $0$ as $m\to \infty$.

We can now state the following slight generalization of Yuan's theorem.
\begin{theorem}[\cite{Yuan}]
Let $K$ be a number field, $X$ be a projective variety over $K$ of dimension $d-1$, $L\to X$ an ample line bundle over $X$ equipped with an adelic semipositive metric. 
Let $Z_m\subset X(\bar{K})$ be any zero-dimensional subvariety defined over $K$ which is generic and small for the height $h_{\bar{L}}$. 

Then, for any place $v\in M_K$,  we have
\begin{equation}\label{eq:yuan}
\frac{1}{\card(Z_m)}\sum_{x\in Z_m}\delta_x \longrightarrow \frac1{\deg (L)} \, c_1(\bar{L})_v^{d-1}
\end{equation}
on $X_v$ in the weak topology of measures.
\label{tmyuan}
\end{theorem}

\section{Transversality of critical orbit relations in $\poly_d$.}\label{sec:trans}

\par This section is devoted to transversality results in the family $\poly_d$ of \emph{all} polynomials. 
This section is an application of Epstein's general theory \cite{epstein2} to our context. 
We follow closely \cite{buffepstein} and adapt it to our situation.


\subsection{The family $\poly_d$ of all polynomials.}
The space $\poly_d$ of all polynomials of degree $d$ is a complex manifold of dimension $d+1$ which is isomorphic to $\C^*\times\C^d$. We denote by $\mathcal{C}(P)\subset \C$ the \emph{critical set} of a polynomial $P$, and by $\mathcal{P}(P)$ its \emph{postcritical set}, i.e.
\begin{center}
$\displaystyle\mathcal{P}(P)=\bigcup_{n\geq1}P^n(\mathcal{C}(P))$.
\end{center}
A simple critical point is a point $c\in \C$ for which $P'(c) =0$ and $P''(c) \neq 0$.
Suppose $P\in\poly_d$ has only simple critical points. Then there exists a neighborhood $V_P\subset\poly_d$ and $(d-1)$ holomorphic functions $c_0,\ldots,c_{d-2}:V_P\longrightarrow\C$ such that $\{c_0(Q),\ldots,c_{d-2}(Q)\}=\mathcal{C}(Q)$ for all  $Q\in V_P$.

~

\par The group $\Aut(\C)= \{ az +b , \, a \in \C^*, \, b \in \C\}$ of affine transformations acts on $\poly_d$ by conjugacy. We shall denote by  $\mathcal{O}(P)$ the orbit of $P\in\poly_d$ under this action. It is a (closed) complex submanifold of $\poly_d$ of dimension $2$.


\subsection{Vector fields and quadratic differentials.}

\paragraph*{Vector fields.}
A tangent vector to $\poly_d$ at $P$ is an equivalence class of holomorphic
maps $\phi : \D \to \poly_d$ such that $\phi(0) = P$ under the relation 
$\phi\sim \psi$ iff $\phi'(0) = \psi'(0)$. The vector space of all tangent vectors at $P$ is denoted by $ T_P\poly_d$.

A tangent vector $\zeta\in T_P\poly_d$ can be identified to a 
section of the line bundle $P^* (T\p^1)$, where $T\p^1$ denotes 
the tangent space of the Riemann sphere. 
Concretely we view $\zeta$ as a holomorphic function  $z \in \C \mapsto \zeta(z) \in T_{P(z)} \C$ that extends to $\infty$ and vanishes there.
To any tangent vector $\zeta\in T_P\poly_d$, we may thus attach a rational
vector field on $\p^1$ with poles included in $\mathcal{C}(P)$:
\begin{center}
$\displaystyle\eta_\zeta(z)\pe  -D_zP^{-1} \cdot \zeta(z) \in  T_{z} \C$.
\end{center}
It vanishes at infinity, and 
when  $P\in\poly_d$ has only simple critical points, then $\eta_\zeta$ has only simple 
poles.

%
%

A vector field on  a finite subset $X\subset\C$ is a collection of tangent vectors $\theta(z) \in T_z\C$ for any $z\in X$. We denote by $\mathcal{T}(X)$ the space of all vector fields on $X$.

Observe that if $\theta$ is a vector field defined on $P(E)$ with $E\subset\C\setminus \mathcal{C}(P)$, then we can define a vector field $P^*\theta$ on $X$ by setting for any $z\in X$
\begin{center}
$P^*\theta(z)\pe D_zP^{-1} \cdot\theta (P(z))$.
\end{center}
\begin{lemma}\label{lem:vf}
Let $\theta$ be any holomorphic vector field defined in a neighborhood of $P(c)\in P(\mathcal{C}(P))$.
Then $P^*\theta$ is a meromorphic vector field in a neighborhood of $c$, and  $\theta(c) =0$ iff $\theta( P(c)) = 0$. When $c$ is a simple critical point, then  $P^*\theta$ has at most a simple pole, and its constant and polar parts  only depend  on $\theta(P(c))$.
\end{lemma}
\begin{proof}
In suitable coordinates at $c$ and $P(c)$, we may write $w= P(z) = z^k$.
If $\theta (w) = (a + wb(w))\frac{\partial}{\partial w}$, then 
$P^*\theta (z) = \frac1{k}(a + z^kb(z^k))z^{1-k}\frac{\partial}{\partial z}$. When $c$ is simple, then $k=2$ and the result follows.
\end{proof}

\paragraph*{Quadratic differentials.}
Recall that a quadratic differential is locally given by $a(z) dz^2$ with $a$ holomorphic.
For any finite subset $X\subset\C$, we denote by $\mathcal{Q}(X)$ the space of meromorphic quadratic differentials on $\p^1$ with at most simple poles in $X \cup \{ \infty\}$.
It follows from Riemann-Roch that $\dim\mathcal{Q}(X)=\max\{\card(X)-2,0\}$.

 If $q\in\mathcal{Q}(X)$ and $\theta$ is a holomorphic vector field defined in a neighborhood of $x\in X$, the product $q\otimes \theta\pe q(\theta,\cdot)$ is well-defined as a meromorphic $1$-form. 

Now pick $q\in\mathcal{Q}(X)$ and any vector field $\tau\in\mathcal{T}(X)$, and choose a holomorphic vector field $\theta$ in a neighborhood of $X$ such that $\theta|_X = \tau$. 
Since $q$ has at most simple poles by definition, the residue at $x$ 
of $q \otimes \theta$ only depends  on $\theta(x)$. We  can thus define the following pairing:
\begin{eqnarray*}
\langle q,\tau\rangle\pe\sum_{x\in X}\textup{R\'es}_x(q\otimes\tau)=\sum_{x\in X}\textup{R\'es}_x(q\otimes\theta)
\end{eqnarray*}

Any postcritically finite polynomial $P\in\poly_d$ induces an operator $P_*$  on $\mathcal{Q}(\mathcal{P}(P))$. For any quadratic differential $q\in\mathcal{Q}(\mathcal{P}(P))$, for any $w\in\C$ and any $x_1,x_2\in T_{w}\p^1$,
 we set
\begin{eqnarray*}
(P_*q)_w(x_1,x_2)\pe\sum_{P(z)=w}q_z\left( D_zP^{-1} \cdot x_1, D_zP^{-1} \cdot x_2\right).
\end{eqnarray*}
In this way we obtain a meromorphic quadratic differential $P_*q$ on $\p^1$, and 
it is not difficult to check that $P_*q\in\mathcal{Q}(\mathcal{P}(P))$.

\par 
A key result from~\cite{buffepstein,epstein2} states that 
the linear operator
\begin{center}
$\nabla_P\pe \id-P_* :\mathcal{Q}(\mathcal{P}(P))\longrightarrow\mathcal{Q}(\mathcal{P}(P))$,
\end{center}
is \emph{bijective}.

\subsection{Guided vector fields.}
\par Following \cite{buffepstein}, we characterize those tangent vectors in $T_P\poly_d$ that 
are tangent to $\O(P)$. First we introduce the notion of guided vector fields. 
\begin{definition}\label{def:guided}
We say that a vector field $\tau\in\mathcal{T}(\mathcal{P}(P))$ is \emph{guided} by $\zeta\in T_P\poly_d$ if 
\begin{center}
$\tau=P^*\tau +\eta_\zeta$ on $\mathcal{P}(P)$, \ and \ $\tau\circ P=\zeta$ on $\mathcal{C}(P)$.
\end{center}
\end{definition}
Some explanations are in order. The equality $\tau\circ P=\zeta$ means that $\tau (P(z)) = \zeta(z)$ in $T_{P(z)}\C$ for all $z\in \mathcal{C}(P)$. 
The vector field $P^*\tau$ is well-defined at each point $z \in \mathcal{P}(P) \setminus  \mathcal{C}(P)$. At a point $z \in \mathcal{P}(P) \cap  \mathcal{C}(P)$, it follows from
Lemma~\ref{lem:vf} that the constant term and the polar part of $P^*\tau$ is well-defined. 
The equality $\tau=P^*\tau +\eta_\zeta$ says that the constant term and the polar parts
of both terms are identical.

Our aim is to show

\begin{proposition}
Let $P\in\poly_d$ be a postcritically finite polynomial with simple critical points, and not conjugated to $z \mapsto z^2$. Pick $\zeta\in T_P\poly_d$.

Then $\zeta\in T_P\mathcal{O}(P)$ iff there exists a vector field
$\tau\in\mathcal{T}(\mathcal{P}(P))$ that is guided by $\zeta$.
\label{propguided}
\end{proposition}
\begin{proof}
Suppose $\zeta\in T_P\mathcal{O}(P)$, and pick a holomorphic map
$\psi: \D \to \Aut(\C)$ with $\psi_t (0) = \id$ such that $\phi'(0) = \zeta$ with  
$\phi= \psi^{-1} \circ P \circ \psi$. Write $\xi \pe \psi'(0) $. 
A direct computation yields
$$
\zeta(z) = D_zP\cdot  \xi(z)  - \xi ( P (z))
$$
At a critical point this equation implies $\xi \circ P = \zeta$.
Pulling back by $P$, we also get $\xi=P^*\xi +\eta_\zeta$ everywhere on $\p^1$.
We conclude by setting $\tau \pe \xi|_{\mathcal{P}(P)}$.

\medskip

For the converse statement we shall rely on the 
\begin{lemma}
Let $\theta$ be any holomorphic vector field defined in a 
neighborhood of $\mathcal{P}(P)$. 
If $\tau \pe \theta|_{\mathcal{P}(P)}$ is guided by 
some $\zeta \in T_P \poly_d$, then $P^*\theta + \eta_\zeta$ is holomorphic
near $\mathcal{C}(P)$.
\label{lmprelim}
\end{lemma}

\begin{lemma}
Let $P\in\poly_d$ be postcritically finite with simple critical points. Assume that $\tau\in\mathcal{T}(\mathcal{P}(P))$ is guided by $\zeta\in T_P\poly_d$. Then, for all $q\in\mathcal{Q}(\mathcal{P}(P))$, we have
\begin{center}
$\langle \nabla_P q,\tau\rangle=0$.
\end{center}
\label{lmdiffquad}
\end{lemma}
Since $\nabla_P$ is invertible, this Lemma ensures that $\langle q,\tau\rangle=0$ for all $q\in\mathcal{Q}(\mathcal{P}(P))$. Extend $\tau$ to a vector field on $\mathcal{P}(P) \cup \{ \infty\}$ by setting $\tau(\infty ) = 0$. Then $\langle q,\tau\rangle=0$   continues to hold for all $q\in\mathcal{Q}(\mathcal{P}(P))$, and \cite[Lemma 7]{buffepstein} implies that  $\tau$
is the restriction of a globally defined holomorphic vector field $\theta$  to $\p^1$ that vanishes at $\infty$. 

Since $\tau$ is guided by $\zeta$, Lemma~\ref{lmprelim} implies that 
$P^*\theta+\eta_\zeta$ is holomorphic on $\p^1$ and vanishes at $\infty$.
When $\mathcal{P}(P)$ has at least $2$ distinct points, $\theta=P^*\theta+\eta_\zeta$ on $\mathcal{P}(P)\cup \{ \infty\}$ implies the equality everywhere on $\p^1$.
We then conclude by applying~\cite[Proposition 1]{buffepstein}.
Since $P$ is supposed to have only simple critical points, $\card(\mathcal{P}(P)) =1$
implies $d =2$ and the critical point is fixed. Whence $P$ is conjugated to $z^2$.
\end{proof}

\begin{proof}[Proof of Lemma~\ref{lmprelim}]
Since $c$ is a simple critical point, we can choose coordinates $z$ at $c \in \mathcal{C}(P)$ and $w$ at $P(c)$ such that  $w = P(z) = z^2$. 
Write $\zeta(z) = (a + O(z)) \frac{\partial}{\partial w}$, so that 
$\eta_\zeta(z) = - D_zP^{-1}\cdot \zeta(z) = (- \frac{a}{2z} + O(z)) \frac{\partial}{\partial z}$. 
Since $\tau = \theta|_{\mathcal{P}(P)}$ is guided by $\zeta$, we have $\theta (P(c)) = \tau ( P(c)) = \zeta(c) = a\frac{\partial}{\partial w}$. Whence  $\theta (w) = (a + O(w) ) \frac{\partial}{\partial w}$, and 
$P^*\theta (z) = \frac12 ( a z^{-1}  + O(z) ) \frac{\partial}{\partial z}$, see the proof of Lemma~\ref{lem:vf}. We conclude that
$P^*\theta (z) + \eta_\zeta= O(z)  \frac{\partial}{\partial z}$ is holomorphic.
\end{proof}

\begin{proof}[Proof of Lemma~\ref{lmdiffquad}]
Let $\theta$ be a holomorphic vector field defined in a neighborhood of $\mathcal{P}(P)$  which coincides with $\tau$ on $\mathcal{P}(P)$. Let $q\in\mathcal{Q}(\mathcal{P}(P))$, then
\begin{eqnarray*}
\langle P_* q,\tau\rangle & = & \sum_{x\in\mathcal{P}(P)}\textup{R\'es}_x((P_*q)\otimes \tau) = \sum_{x\in\mathcal{P}(P)}\textup{R\'es}_x((P_*q)\otimes \theta).
\end{eqnarray*}
Since $\eta_\zeta$ is  meromorphic on $\p^1$ with poles in
$\mathcal{C}(P)$, the $1$-form $q\otimes \eta_\zeta$ is also meromorphic on $\p^1$ with poles in   $\mathcal{P}(P)\cup \mathcal{C}(P)$ and  the sum of its residues  vanishes. Applying  the change of variable formula, we get
\begin{eqnarray}
\langle P_* q,\tau\rangle=\sum_{x\in\mathcal{P}(P)\cup\mathcal{C}(P)}\textup{R\'es}_x(q\otimes P^*\theta) =  \sum_{x\in\mathcal{P}(P)\cup\mathcal{C}(P)}\textup{R\'es}_x(q\otimes (P^*\theta+\eta_\zeta)).
\label{diffquad}
\end{eqnarray}
Now, since $\tau$ is guided, $P^*\theta+\eta_\zeta=P^*\tau+\eta_\zeta=\tau$ on $\mathcal{P}(P)$, and $P^*\theta+\eta_\zeta$ is holomorphic in a neighborhood of 
$\mathcal{C}(P)$ by Lemma~\ref{lmprelim}.
This gives $\langle P_* q,\tau\rangle=\langle q,\tau\rangle$, which ends the proof.
\end{proof}

\subsection{Transversality at strictly postcritically finite parameters.}
Pick any strictly postcritically finite polynomial $P$ with simple critical points, and choose a neighborhood $V_P\subset \poly_d$ of $P$  with holomorphic functions 
 $c_0, \ldots , c_{d-2}: V_P \to \C$ such that $\{ c_0(Q) , \ldots , c_{d-2}(Q) \} = \mathcal{C}(Q)$ for all $Q \in V_P$.

 For any $0\leq i\leq d-2$, choose $m_i>n_i\geq1$ such that
\begin{center}
$P^{m_i}(c_i(P))=P^{n_i}(c_i(P))$.
\end{center}
Observe that this implies the point $P^{n_i}(c_i(P))$ to be a periodic point for $P$ of period dividing $m_i-n_i$. We thus define the following holomorphic maps $\mathfrak{n},\mathfrak{m}:V_P \subset \poly_d\longrightarrow \C^{d-1}$: 
\begin{center}
$\mathfrak{n}(Q)\pe(Q^{n_0}(c_0),\ldots,Q^{n_{d-2}}(c_{d-2}))$ and $\mathfrak{m}(Q)\pe(Q^{m_0}(c_0),\ldots,Q^{m_{d-2}}(c_{d-2}))$,
\end{center}
and  adapt the arguments of \cite{buffepstein} to prove the following
\begin{theorem}
Suppose $P\in\poly_d$ is strictly postcritically finite with simple critical points. 
Assume moreover that 
\begin{enumerate}
\item
the orbits of the critical points are disjoint: for any two critical points $c_i\neq c_j$ then 
$P^k(c_i) \neq P^l(c_j)$ for all $k,l \ge 0$;
\item
for each $i$, $P^{n}(c_i)$ is periodic iff $n\ge n_i$ and its period is then exactly equal to $m_i-n_i$.
\end{enumerate}
Then we have
\begin{center}
$\ker(D_P\mathfrak{n}-D_P\mathfrak{m})=T_P\mathcal{O}(P)$.
\end{center}
\label{tmAdam}
\end{theorem}

\begin{corollary}\label{corAdam}
Under the same assumptions as in the previous theorem,  the $(d-1)$ local hypersurfaces $\{ Q, \, Q^{m_i}(c_i) = Q^{n_i}(c_i)\}_{0\le i \le d-2}$ are smooth at $P$ and transversal.
\end{corollary}

\begin{proof}[Proof of Corollary \ref{corAdam}]
The dimension of $V_P$ is equal to $d+1$, and the dimension of $T_P\mathcal{O}(P)$ is equal to $2$. By the previous theorem the map $\mathfrak{n} - \mathfrak{m} : V_P \to \C^{d-1}$
has maximal rank equal to $\dim (V_P)-2$. It follows that there exists coordinates at $P$ in $V_P$ such that $\mathfrak{n} - \mathfrak{m}$ is a linear projection map. In this coordinate system, the hypersurfaces are coordinates hyperplanes, and are thus smooth and transversal.
\end{proof}

To simplify notation, we write $\dot u\pe\displaystyle\left. \frac{du_t}{dt}\right|_{t=0}$
for any differentiable map $t \mapsto u_t$.

\begin{proof}[Proof of Theorem \ref{tmAdam}]
It is clear that $T_P\mathcal{O}(P)\subset\ker(D_P\mathfrak{n}-D_P\mathfrak{m})$, so pick  $\zeta\in\ker(D_P\mathfrak{n}-D_P\mathfrak{m})$. Proposition \ref{propguided} guarantees that it suffices to prove that there exists $\tau\in\mathcal{T}(\mathcal{P}(P))$ which is guided by $\zeta$ to conclude.
\par Since $\zeta\in T_P\poly_d$, there exists a holomorphic disc $t\mapsto P_t\in\poly_d$ with $P_0=P$ and such that $\dot P=\zeta$. To simplify notation we shall write
\begin{center}
$c_{i,t}\pe c_i(P_t)$ and $v_{n,i,t}\pe P_t^n(c_{i,t})$ for $n\geq0$.
\end{center}
We shall also  let $c_i \pe c_{i,0}$, and $v_{n,i} = v_{n,i,0}$.

\medskip

We first define a vector field $\tau\in\mathcal{T}(\mathcal{P}(P))$ and then check that it is guided by $\zeta$. To do so we pick a critical point $c_i$ of $P$ and define $\tau$ on the orbit of $c_i$ by setting:
$$\tau(v_{n,i}) \pe \dot v_{n,i}
\text{ for any } 1 \le n  < m_i
~.$$
Observe that $\tau$ is well-defined on $\mathcal{P}(P)$ since by assumption (1) all critical orbits are disjoint, and by assumption (2) $P^n(c_i) \neq P^{n'}(c_i)$ for all $ n \neq n' < m_i$.

\medskip

It remains to check that $\tau$ is guided by $\zeta$.
A first observation is that $\tau (P(c_i)) =  \zeta (c_i)$ by definition.
It is thus only necessary to check the equality 
\begin{equation}\label{eqguided}
\tau = P^*\tau + \eta_\zeta
\end{equation}
on $\mathcal{P}(P)$.

Since $\zeta\in\ker(D_P\mathfrak{n}-D_P\mathfrak{m})$ and $v_{n_i,i}=v_{m_i,i}$, we have
\begin{center}
$\dot v_{m_i,i}=D_P\mathfrak{m}\cdot \zeta=D_P\mathfrak{n}\cdot\zeta=\dot v_{n_i,i}=\tau(v_{n_i,i})=\tau(v_{m_i,i})$.
\end{center}

For $ 1\le n \le m_i$, we have
\begin{center}
$\tau(v_{n+1,i})=
\left. \frac{d}{dt}\right|_{t=0} P_t(v_{n,i,t}) = 
(\left. \frac{d}{dt}\right|_{t=0} P_t)(v_{n,i})+D_{v_{n,i}}P \cdot \dot v_{n,i} = \zeta(v_{n,i})+D_{v_{n,i}}P\cdot\tau(v_{n,i})$.
\end{center}
Since the point $v_{n,i}$ is not a critical point of $P$, applying $(D_{v_{n,i}}P)^{-1}$ gives $P^*\tau(v_{n,i})=-\eta_\zeta(v_{n,i})+\tau(v_{n,i})$ for all $n\ge1$. This concludes the proof.
\end{proof}

\begin{remark}
Recall that the periodic points contained in the critical orbits of a strictly postcritically finite polynomial are repelling. Levin \cite{Levin3} proved recently a similar transversality result for maps satisfying a weaker expansivity property along their critical orbits.
\end{remark}

\subsection{Transversality at hyperbolic parameters.}
Pick any postcritically finite hyperbolic polynomial $P$ with periodic simple critical points. Choose a neighborhood $V_P\subset \poly_d$ of $P$  and holomorphic functions $c_0, \ldots , c_{d-2}: V_P \to \C$ such that $\{ c_0(Q) , \ldots , c_{d-2}(Q) \} = \mathcal{C}(Q)$ for all $Q \in V_P$.
 For any $0\leq i\leq d-2$, choose $m_i \ge1$ such that
\begin{center}
$P^{m_i}(c_i(P))=c_i(P)$.
\end{center}
Observe that this means the critical point $c_i(P)$ to be a periodic point for $P$ of period dividing $m_i$. We then define the following holomorphic maps $\mathfrak{c},\mathfrak{m}:V_P \subset \poly_d\longrightarrow \C^{d-1}$: 
\begin{center}
$\mathfrak{c}(Q)\pe(c_0,\ldots,c_{d-2})$ and $\mathfrak{m}(Q)\pe(Q^{m_0}(c_0),\ldots,Q^{m_{d-2}}(c_{d-2}))$.
\end{center}
 One can  adapt the arguments of the previous section to prove the following.
\begin{theorem}
Let $P\in\poly_d$ be a hyperbolic postcritically finite with periodic simple critical points.
Then we have
\begin{center}
$\ker(D_P\mathfrak{c}-D_P\mathfrak{m})=T_P\mathcal{O}(P)$.
\end{center}
\label{tmAdam2}
\end{theorem}
The same proof as for Corollary~\ref{corAdam} yields
\begin{corollary}\label{corAdam2}
If $P$ has simple critical points, then the $(d-1)$ local hypersurfaces $\{ Q\in\poly_d, \, Q^{m_i}(c_i) = c_i\}$, $0\le i \le d-2$ are smooth at $P$ and transversal.
\end{corollary}

\begin{proof}[Proof of Theorem \ref{tmAdam2}]
The case $d=2$ and $P(z) = z^2$ has to be treated separately. 
In this case $Q(z) = b_2z^2 + b_1 z+ b_0$ and 
$c(Q)= - \frac{b_1}{2b_2}$ hence $D_P \mathfrak{c}  (h_0, h_1, h_2)= -h_1/2$. 
On the other hand, for $Q= z^2 + b_0$ then $\mathfrak{m}(Q) = Q^n(0) = b_0 + O(b_2^0)$
hence $D_P \mathfrak{m}  (0, 0, 1)= 1$. It follows that the linear form
$D_P \mathfrak{c} -  D_P \mathfrak{m}$ is non zero. Since 
$\ker(D_P\mathfrak{c}-D_P\mathfrak{m})\supset T_P\mathcal{O}(P)$
and the latter space has dimension $2$, we conclude to the equality.

\medskip

In the remaining cases we may and shall apply Proposition~\ref{propguided}.
As in the proof of Theorem~\ref{tmAdam}, we pick $\zeta\in\ker(D_P\mathfrak{c}-D_P\mathfrak{m})$, and choose a holomorphic disc $t\mapsto P_t\in\poly_d$ with $P_0=P$ and such that $\dot P=\zeta$.
Again, we write
\begin{center}
$c_{i,t}\pe c_i(P_t)$ and $v_{n,i,t}\pe P_t^n(c_{i,t})$ for $n\geq0$.
\end{center}
We shall also  let $c_i \pe c_{i,0}$, and $v_{n,i} = v_{n,i,0}$. Recall that for any $n\ge 0$, we have the relation
\begin{eqnarray}
\dot v_{n+1,i}=
 \zeta(v_{n,i})+D_{v_{n,i}}P\cdot\dot v_{n,i}~.
\label{eq:iterate}
\end{eqnarray}
We shall deduce from this equation the following
\begin{lemma}\label{lmdirect}
~
For any $n,m \ge 0$, and for all $i,j$ such that $v_{n,i} = v_{m,j}$, we have $\dot v_{n,i} = \dot v_{m,j}$.
\end{lemma}
Taking this result for granted , we continue with the definition of a vector field on $\mathcal{P}(P)$ that is guided by $\zeta$.
Pick any point $x\in \mathcal{P}(P)$,  choose $i$ and $n\ge 1$ such that $x = v_{n,i}$, and
define $\tau(x) \pe \dot v_{n,i}$. The previous lemma shows that $\tau$ is well-defined at $x$
independently on the choice of integers $n,i$ such that  $x = v_{n,i}$.

To conclude it remains to check that $\tau$ is guided by $\zeta$. The equality $\tau (P(c_i)) = \zeta (c_i)$
follows from the definition. When $x=v_{n,i}$ is not a critical point, then  applying $D_xP^{-1}$ to~\eqref{eq:iterate} gives $\tau = P^*\tau + \eta_\zeta$ at $x$.

When $x = c_i$ is a critical point,
we need to be more careful since $D_xP =0$.
Since $x$ is a simple critical point, we may choose coordinates $z$ at $c_i$ and $w$ at $P(c_i)$ such that
$w = P_t(z) = z^2+ t (a + O(z)) + O(t^2)$. Since we may follow the critical point for $t$ small, we may also suppose that $c_{i,t} (z) = 0$ for all $t$
so that $P_t(z) = z^2 + t (a + O(z^2)) + O(t^2)$.
As in the proof of Lemma~\ref{lmprelim}, we obtain 
$\zeta (z) = ( a + O(z)) \frac{\partial}{\partial w}$, and
$\eta_\zeta(z) = ( -\frac{a}{2z} + O(z))\frac{\partial}{\partial z}$.
Observe that in our coordinates we have 
$\tau(c_i) = \left. \frac{d}{dt}\right|_{t=0} c_{i,t} = 0$, and 
$\tau(P(c_i)) = \left. \frac{d}{dt}\right|_{t=0} P_t(c_{i,t}) =a\,\frac{\partial}{\partial z}$.
We may thus extend $\tau$ locally at $c_i$ and $P(c_i)$ holomorphically 
by setting
$\tau (z) \equiv 0$ and $\tau(w) \equiv a$.
It follows that 
$$
P^*\tau (z) + \eta_\zeta (z) - \tau (z) 
= 
\frac{a}{P'(z)} \,\frac{\partial}{\partial z}+ \left( -\frac{a}{2z} + O(z)\right) \frac{\partial}{\partial z} - 0 = O(z) \frac{\partial}{\partial z}
~.
$$
From the discussion after Definition~\ref{def:guided}, it follows that 
$P^*\tau + \eta_\zeta = \tau $ at any critical point.
This concludes the proof.
\end{proof}

\begin{proof}[Proof of Lemma~\ref{lmdirect}]
Fix $i$ for a moment, and to simplify notation
write $v_k, q,m, c$ instead of $v_{k,i}, q_i, m_i, c_i$ respectively.
Recall that $q$ is the exact period of $v_0 = c$. 
For any $l \ge 1$, iterating the assertion \eqref{eq:iterate} and using the fact that $DP^q$ is vanishing at all points of the cycle and that $v_{k+q} = v_k$ for all $k\ge 0$, 
 give
\begin{eqnarray*}
\dot v_{lq} 
\!\!& = &\!\!
\zeta(v_{lq -1}) + D_{v_{lq -1}}P \cdot  \zeta(v_{lq -2}) + \ldots + D_{v_{(l-1)q +1}}\!P^{q-1} \cdot  \zeta(v_{(l-1)q}) + D_{v_{(l-1)q}}P^{q} \cdot  \dot v_{(l-1)q}
\\
\!\!& = &\!\!
\zeta(v_{lq -1}) + D_{v_{lq -1}}P \cdot  \zeta(v_{lq -2}) + \ldots + D_{v_{(l-1)q+1}}P \cdot  \zeta(v_{(l-1)q}) 
\\
\!\!& = &\!\!
\zeta(v_{q -1}) + D_{v_{q -1}}P \cdot  \zeta(v_{q -2}) + \ldots + D_{v_{1}}P \cdot  \zeta(v_0) 
= 
\dot v_{q} 
\end{eqnarray*}
Since $\zeta\in\ker(D_P\mathfrak{c}-D_P\mathfrak{m})$, we also have
$$\dot v_{0} = \dot c=D_P\mathfrak{c}\cdot \zeta=D_P\mathfrak{m}\cdot \zeta=\dot v_{m}~,$$
whence $\dot v_{lq} = \dot v_{0}$ for all $l\ge 1$ since $m$ is divisible by $q$.
Again by~\eqref{eq:iterate} we get 
$$\dot v_{lq+1} = \zeta (v_{lq}) + D_{v_{lq}} P \cdot v_{lq} = 
 \zeta (v_{0}) + D_{v_{0}} P \cdot v_{0} = 
\dot v_{1}~.$$
An immediate induction on $k\ge 0$ then proves $\dot v_{lq + k} = \dot v_{k}$ for all $l\ge 0$.
This proves the lemma in the case $i=j$.

\smallskip

Assume now that $v_{k,i}=c_j$ for some $j\neq i$ and some $k\ge1$. Observe that $q \pe q_i = q_j$.
Since $P$ has only simple critical points, we may assume that $1\le k\le q-1$, and $k$ is then uniquely determined. 
By \eqref{eq:iterate}, we get
$$\dot v_{k+1,i}= \zeta(v_{k,i})+D_{v_{k,i}}P\cdot \dot v_{k,i}=\zeta(c_j) = \dot v_{1,j}~.$$
Again by \eqref{eq:iterate} it follows by induction that 
$\dot v_{k+m,i}= \dot v_{m,j}$ for all $m\ge 1$.

Now suppose $v_{n,i} = v_{m,j}$. Permuting $i$ and $j$ if necessary we may assume that
$n = k+ m  + ql$ for some $l\ge 0$, and we have
$$\dot v_{n,i}= 
\dot v_{k+m+ql,i}=
\dot v_{k+m+q,i}=\dot v_{m+q,j} = 
 \dot v_{m,j}~,$$
which ends the proof.
\end{proof}


\section{Distribution of strictly postcritically finite parameters.}\label{sec:distrib}

The aim of this section is to prove Theorem~\ref{maintm1}. Let us first set some notation.
For $m,n\geq0$ and $0\leq i \leq d-2$, we let
\begin{eqnarray*}
\Per_i(m,n)
& \pe &
\left\{(c,a)\in\C^{d-1} , \ P_{c,a}^m(c_i)=P_{c,a}^n(c_i)\right\}~, 
\\
\Per_i^*(m,n)
& \pe &
\left\{(c,a)\in\C^{d-1} ,\  P^k(c_i) \text{ is periodic iff } k\ge n, \text{ with period exactly } m-n\right\}, 
\end{eqnarray*}
Obviously $\Per_i^*(m,n)\subset \Per_i(m,n)$.
Given any $(d-1)$-tuples $\um = (m_0, \ldots , m_{d-2}), \un = (n_0, \ldots, n_{d-2})$ of non-negative integers, we also set
\begin{center}
$\displaystyle\PCF (\um,\un)\pe \bigcap_{i=0}^{d-2} \Per_i(m_i,n_i) \text{ and } 
\displaystyle\PCF^* (\um,\un)\pe \bigcap_{i=0}^{d-2} \Per_i^*(m_i,n_i) 
~.$
\end{center}
Observe that any polynomial in $\PCF (\um,\un)$ is post-critically finite. 
We  define
$$
\TPCF (\um,\un) := \{ ( a,c)\in \PCF (\um,\un),\,  \Per_i(m_i,n_i) \text{ are smooth and transverse at } (c,a)\}~.
$$
We finally write $$|\um| = m_0+ \ldots + m_{d-2}.$$


\subsection{Transversality in the family $P_{c,a}$.}

The next result is crucial to our analysis.  It is a direct consequence of Theorem~\ref{tmAdam}.

\begin{theorem}\label{cortransverse}
Let $\um, \un$ be two $(d-1)$-tuples of integers such that $m_j>n_j>0$ and let $(c,a)\in\PCF(\um,\un)$ be such that $P_{c,a}$ has only simple critical points. Suppose 
\begin{enumerate}
\item
the orbits of the critical points are disjoint: for any two critical points $c_i\neq c_j$ then 
$P^k(c_i) \neq P^l(c_j)$ for all $k,l \ge 0$;
\item
for each $i$, $P^{n}(c_i)$ is periodic iff $n\ge n_i$ and its period is then exactly equal to $m_i-n_i$.
\end{enumerate}
Then the $(d-1)$ hypersurfaces $\Per_j(m_j,n_j)$ are smooth and intersect transversely at the point $(c,a)$.
\end{theorem}

We rely on 
\begin{proposition}
The set $\Lambda\pe\{P_{c,a}\in\poly_d  /  (c,a)\in\C^{d-1}\}$ is smooth subvariety of $\poly_d$ of dimension $d-1$. Moreover, if $P_{c,a}$ has simple critical points, then $\Lambda$ and $\mathcal{O}(P_{c,a})$ intersect transversely at $P_{c,a}$.
\label{proptransverse}
\end{proposition}

\begin{proof}[Proof of Theorem~\ref{cortransverse}]
The theorem follows from Corollary~\ref{corAdam} and the following general fact.
Suppose we have $k$ coordinate hyperplanes  $H_1, \ldots, H_k$ in $\C^n$, and
pick any smooth subvariety $\Lambda$ that is transversal to the intersection $H_1\cap \ldots \cap H_k$.
Then the hyperplanes $H_i \cap \Lambda$ are smooth and have transversal intersections in $\Lambda$.
\end{proof}

\begin{proof}[Proof of Proposition~\ref{proptransverse}]
Let us parameterize $\poly_d$ by $P(z) = \sum b_i z^i$ with $(b_0, \ldots , b_d) \in \C^{d}\times \C^*$. Then the space $\Lambda\pe \{P_{c,a}\in\poly_d , \  (c,a)\in\C^{d-1}\}$
is determined by the equations $\{ b_d = \frac1d, \, b_1 = 0 \}$ and is thus clearly smooth.

The space $T_P \O(P)$ is two-dimensional and generated by
$$
\zeta_1 = \left. \frac{d}{dt}\right|_{t=0} P( z + t) - t = \sum_1^{d} i b_i z^{i-1} - 1
$$
and
$$
\zeta_2= \left. \frac{d}{dt}\right|_{t=1}  \frac1t\, P(t z)  = \sum_2^d (i-1) b_i z^i~.
$$
Suppose $a_1 \zeta_1 + a_2 \zeta_2 \in T_P\Lambda$ for some $a_1, a_2 \in \C$.
Since $T_P\Lambda = \{ \sum_0^d \beta_i z^i, \, \beta_d = \beta_1 =0\}$, we have
that $(d-1) a_2b_d = 0$, and $a_1 (2b_2) =0$. The degree of $P$ is equal to $d$, hence $b_d \neq 0$ and $a_2 =0$. Recall that $P'(0) = 0$ when $P\in \Lambda$.
When $P$ has only simple critical points,  then $b_2 \neq 0$ and
$a_1 =0$. This proves $T_P \O(P) \cap T_P\Lambda = (0)$.\end{proof}


\subsection{Bounds on postcritically finite polynomials in a fixed subvariety.}

\par 

\begin{theorem}\label{thmupperbound}
Pick any two sequences of $(d-1)$-tuples $(\um_k,\un_k)$ of non-negative integers such that 
$m_{k,i}> n_{k,i}$ and $m_{k,i}\rightarrow\infty$ for all $0\le i\le d-2$. For any proper algebraic subvariety $V\subset\A^{d-1}_\C$, we have
\begin{center}
$ \card \left( V\cap \TPCF(\um_k,\un_k) \right) = o (d^{|\um_k|})~.$
\end{center}
\end{theorem}

\begin{proof}[Proof of Theorem~\ref{thmupperbound}]
We shall rely on the next two lemmas.
\begin{lemma}\label{lmtransversal}
Suppose $V$ is any irreducible algebraic subvariety of dimension $q$ in $\A^{d-1}_\C$, and let $p$ be a smooth point in $V$. 
 Pick a finite set of hyperplanes $H_1, \ldots , H_{d-1}$ that intersect transversely at $p$ in $\A^{d-1}_\C$.

Then there exists a finite set $I\subset \{ 1, \ldots, q-1\}$ of cardinality $q$ such that 
$p$ is an isolated point of $V\cap_{i \in I}  H_i$.
\end{lemma}

\begin{lemma}\label{lmdegree}
For any $0\le i \le d-2$ and any integer $m> n\ge0$ we have
$$
\deg ( P^m_{c,a}(c_i) - P^n_{c,a}(c_i))= d^m~.
$$
\end{lemma}
For any multi-index $|I| = q$, we decompose into two subsets
$$
V \cap \TPCF_I (\um_k , \un_k) = F_{I,k} \cup Z_{I,k}
$$
where $F_{I,k}$ consists of those isolated points of $V\cap \TPCF_I(\um_k,\un_k)$, and $Z_{I,k}$ is the union of all positive dimensional components of $V\cap \TPCF_I(\um_k,\un_k)$.
Observe that by Bezout's Theorem and Lemma~\ref{lmdegree}, we have
\begin{equation}\label{eqbdbezout}
\card(F_{I,k}) \le \deg (V) \, \prod_{i\in I} \deg (\Per_i(m_{k,i},n_{k,i}) ) = \deg(V) \, d^{\sum_{i\in I} m_{k,i}}~.
\end{equation}
Let $\reg(V)$ be the regular locus of $V$. It is an open Zariski dense subset of $V$.
By Lemma~\ref{lmtransversal}, for any point $p\in \TPCF (\um_k,\un_k)\cap \reg(V)$ one can find a multi-index $|I| = q$ such that $p\in F_{I,k}$. It follows that
$$
\card \left( \reg(V) \cap \TPCF(\um_k,\un_k) \right)
\le 
\sum_{|I| =q} \card F_{I,k}
\le 
\sum_{|I| =q} \deg(V) \, d^{\sum_{i\in I} m_{k,i}}~.
$$
Since $V$ is a proper subvariety, $q< d-1$ whence we can find a constant $C>0$ only depending on $V$, $q$ and $d$ such that 
$$
\card \left( \reg(V) \cap \TPCF(\um_k,\un_k) \right)
\le C\,  d^{|\um_k| - \min_{0\le i \le d-2} m_{k,i}}~.
$$
The latter bound is $o(d^{|\um_k|})$ since $m_{k,i} \to \infty$ for all $i$.

The result follows by stratifying $V = \reg (V) \cup \reg ( \sing (V)) \cup \reg ( \sing ( \sing (V)))\cup \ldots$, and by applying the preceding bound to each strata. 
\end{proof}

\begin{proof}[Proof of Lemma~\ref{lmtransversal}]
In suitable coordinates $z_i$ at $p$, we may suppose that $H_i = \{ z_i = 0\}$ for each $i$.
Since $V$ is smooth, its tangent space at $p$ has dimension $q$ and one of the $q$-forms
$\om_I:= dz^I$ for $|I| =q$ satisfies $\om_I|_{T_pV} \neq 0$. Since the kernel of $\om_I$ is the tangent space of $\cap_I H_i$ at $p$, we conclude that the intersection between $V$ and the $H_i$'s with $i\in I$ is transversal.
\end{proof}

\begin{proof}[Proof of Lemma~\ref{lmdegree}]
An immediate induction shows that 
$$ P^{l+1}_{c,a} (c_i) = \frac1{d^{1+\ldots + d^{l-1}}}\, P_{c,a}(c_i) ^{d^l}  + Q (c,a) ~,$$
where $Q$ is a polynomial of degree $< d^{l+1}$.
\end{proof}


\subsection{Lower bound for the number of postcritically finite polynomials.}

We shall prove the following result.
\begin{theorem}\label{thm:countSPCF}
There exists a positive constant $c>0$ depending only on $d$ such that, if $(\um_k,\un_k)$ are any two sequences of $(d-1)$-tuples of non-negative integers such that 
$m_{k,i}> n_{k,i}\ge 1$ and $m_{k,i}\rightarrow\infty$ for all $0\le i\le d-2$, there exists $k_0\ge1$ such that
\begin{center}
$\displaystyle \card \left(\PCF^*(\um_k,\un_k)\right) \ge 
c\, d^{|\um_k|}$
\end{center}
for all $k\ge k_0$.
\end{theorem}

We rely on a deep result of J.~Kiwi~\cite{kiwi-portrait,BFH}, and follow the exposition of~\cite{favredujardin}. Let us first introduce some notation.
\begin{definition}
  We denote by $\mathsf{S}$ the set of pairs $\{ \alpha , \alpha ' \}$
  contained in the circle $\mathbb{R}/\mathbb{Z} $, such that $d\alpha
  = d \alpha '$ and $\alpha \neq \alpha '$.
\end{definition}
Two finite and disjoint subsets $\theta_1, \theta_2 \subset
\mathbb{R}/\mathbb{Z} $ are said to be \emph{unlinked} if $\theta_2$
is included in a single connected component of $(\mathbb{R}/\mathbb{Z})
\setminus \theta_1$.
\begin{definition}
 We let $\mathsf{Cb}_0$ be the set of $(d-1)$-tuples $\Theta =
  (\theta_1,\cdots , \theta_{d-1})\in \mathsf{S}^{d-1}$ such that for
  all $i\neq j$, the two pairs $\theta_i$ and $\theta_j$ are disjoint
  and unlinked.
\end{definition}

\begin{proof}[Proof of Theorem~\ref{thm:countSPCF}]
Write $\Phi_d(\a) = d \a$ on $\R/\Z$.
By~\cite[Theorem 7.18]{favredujardin}, the subset of $\PCF^*(\um_k,\un_k)$ having only simple critical points is in bijection with the following set
\begin{multline*}
\mathsf{C}_k
:= \{  \Theta \in \mathsf{Cb}_0 ,\, \theta_i = (\a_i, \a'_i),\, \text{ s.t. for all } 1\le i \le d-1, \, 
d^{n} \a_i \text{ is } \Phi_d-\text{periodic iff } n \ge n_{k,i}, \\
\text{ and its period is  precisely equal to  } m_{k,i}- n_{k,i} 
\}~.
\end{multline*}
We thus need to find a positive constant $c>0$ such that
$$
\card (\mathsf{C}_k) \ge c\, d^{|\um_k|}
$$
for all $k\ge k_0$.
The precise count of $\mathsf{C}_k$ is a very delicate issue, but obtaining a lower bound is
much easier.

\begin{lemma}\label{lmcarpmn}
For any $m>n > 0$, define
$$P(m,n) \pe \{ \a \in \R/ \Z, \, d^{n'} \a\text{ is } \Phi_d-\text{periodic iff } n' \ge n, \\
\text{ with period  equal to  } m- n\} ~.$$
There exists a constant $c>0$ such that for  all $x \in \R /\Z$
\begin{equation}\label{eq:victory}
\card \left( P(m,n) \cap \left[x, x + \frac1{d^2}\right] \right)\ge  c\, d^m
\end{equation}
for all  $m > n >0$.
\end{lemma}
\begin{proof}[Proof of Lemma~\ref{lmcarpmn}]
Observe that a point $y\in \R/\Z$ has exactly $d^k$ preimages by $\Phi_d^k$ and that these preimages are
equidistributed on $\R/\Z$. It follows that 
\begin{equation}\label{eqprem}
\card \left( \Phi_d^{-k}\{ y \}  \cap \left[x, x + \frac1{d^2}\right] \right) 
\ge d^{k-3}
\end{equation}
for all $k \ge K$ large enough and for all $x$ and $y$.

The number of periodic points of period divisible by $m-n$ is equal to $d^{m-n} -1$. 
It follows from the M\"obius inversion formula that the number 
$\Per(\Phi_d, m-n)$
of periodic points of period equal to 
$m-n$ can be bounded from below by
$$
\Per(\Phi_d, m-n) =
\sum_{l | (m-n)} \mu \left( \frac{m-n} {l}\right) (d^l -1 )
\ge 
d^{m-n} - \frac{d^{1+ \frac{m-n}{2}} -1}{d-1}~.
$$
Now any point in this cycle admits $(d-1)$ preimages that lie in $P(m-n,1)$, and each of these
points admits $\ge d^{(n-1)-3}$ points in $P(m,n)\cap \left[x, x + \frac1{d^2}\right]$ by~\eqref{eqprem} if $n\ge K+1$.
In this case, we thus get
\begin{multline*}
\card\left( P(m,n) \cap \left[x, x + \frac1{d^2}\right] \right)\ge d^{n-4} (d-1)\, 
\left(d^{m-n} - \frac{d^{1+ \frac{m-n}{2}} -1}{d-1}\right) 
\ge \\
 \frac{d-1}{d^4}\, d^m \left( 1 - \frac{d}{d-1} d^{(n-m)/2}\right)
\ge   \frac{d-1}{d^4}\, d^m \left( 1 - \frac{1}{d-1}\right)
\end{multline*}
at least when $m-n\ge2$. When $m-n=1$, observe that one can directly get the bound
$\card P(m,n) \ge d^{n-4} (d-1)\, d^{m-n}$.
This concludes the proof when $n \ge K+1$. 

\smallskip

When $1 \le n \le K$, we proceed as follows. Pick any closed segment $I \subset \R/\Z$ of sufficiently  length $|I| < 1$. Observe that 
$$
\left| \card ( I \cap \{\a , \, d^k \a = \a \} ) - \frac{d^k -1}{|I|} \right| \le 2
$$
when $k\ge 1$. This proves 
$$\card \left( I \cap \Per(\Phi_d, m-n) \right) \ge \frac{d^{m-n} -1}{|I|}  - \sum_{l\le (m-n)/2} \frac{d^l-1}{|I|} +2
\ge c\, \frac{d^{m-n}}{|I|}$$
for a suitable constant $c>0$. Observe that  we can  chose $c$ to be arbitrarily close to $1$ 
if $m \ge m_0 \gg1$.  Now since  $\Phi_d^n [x , x + d^{-2}]$ contains a segment of length at least $1/d$, for all $m\ge m_0$ 
we conclude that 
\begin{multline*}
\card ( P(m,n) \cap [x , x + d^{-2}] ) \ge 
\\
\card  \left(  \Per(\Phi_d, m-n)\cap  \Phi_d^n [x , x + d^{-2}] \right) 
- 
\card  \left(  \Per(\Phi_d, m-n)\cap [x , x + d^{-2}]\right) 
\\ 
\ge c \, d^{m-n-1} - \left( \frac{d^{m-n}-1}{d^2}+2\right) 
\ge c'\, d^{m-n}
\end{multline*}
which concludes the proof in this case too.
\end{proof}

Fix for a moment $\um,\un$ with $m_i>n_i\ge 1$ for all $i$, and
choose  a random point $\a_1$ in $P(m_1,n_1)$. We have at least $cd^{m_1}$ possibilities. Next choose a point in 
$\a_2\in P(m_2,n_2)$ in such a way that $\a_2\ge \a_1$ in the standard orientation of $\R/\Z$ and
$$
\a_1 + \frac1{d} < \a_2 < \a_1 + \frac1{d-1}~.
$$
By~\eqref{eq:victory}, we have at least $c  d^{m_2}$ possibilities.
We continue inductively constructing $\a_{j+1}\in P(m_{j+1},n_{j+1})$ such that $\a_{j+1}\ge \a_j$ and
$$
a_j + \frac1{d} < \a_{j+1} < \a_j+ \frac1{d-1}~.
$$
We again have at least at least $cd^{m_{j+1}}$ possibilities. We end up with at least $c^{d-1} d^{|\um|}$  possible $(d-1)$-tuples $(\a_1, \ldots , \a_{d-1})$
such that $\a_i \in P(m_i, n_i)$ for all $i$. 

To any $(\a_1, \ldots , \a_{d-1})$ as above
we let $\theta_i = \{ \a_i, \a_i + \frac1{d}\}$. Observe that, by construction, we  have $|\a_i - \a_j| > \frac1{d}$
for all $1\le i\neq j\le d-1$. Hence the $(d-1)$ pairs $\theta_1, \ldots , \theta_{d-1}$ are  unlinked. 

\smallskip

Applying this construction to $\um_k,\un_k$ implies
$\card(\mathsf{C}_k) \ge c^{d-1}\, d^{|\um_k|}$ as required.
\end{proof}


\subsection{Counting postcritically finite polynomials with critical relations.}

\begin{theorem}\label{thmcriticalrelations}
Pick any two sequences of $(d-1)$-tuples $(\um_k,\un_k)$ of non-negative integers such that 
$m_{k,i}> n_{k,i}>0$ and $m_{k,i}-n_{k,i} \rightarrow\infty$ for all $0\le i\le d-2$. 
\begin{equation}\label{eqwouf}
\card \left( \{ (c,a)\in \PCF^*(\um_k,\un_k), \, P_{c,a}^l(c_i) = P_{c,a} ^{l'} (c_j) \text{ for } i\neq j \text{ and } l, l' \ge 0\} \right) = o (d^{|\um_k|})~.
\end{equation}
\end{theorem}

\begin{proof}[Proof of Theorem~\ref{thmcriticalrelations}]
Suppose that there exists a critical relation $P_{c,a}^l(c_i) = P_{c,a} ^{l'} (c_j)$ for some $i\neq j$  and some integers $l, l' \ge 0$. Since $c_i$ and $c_j$ are eventually mapped to the same cycle after $n_{k,i}$ and $n_{k,j}$ iterates respectively,  the minimal integer $h$ (resp. $h'$) such that $P_{c,a}^h(c_i)$ (resp. $P_{c,a}^{h'}(c_j)$) belongs to the orbit of $c_j$ (resp. of $c_i$)
is less than $n_{k,i}$ (resp. than $n_{k,j}$). Permuting $i$ and $j$ if necessary we may assume that
$P^{h"} (P_{c,a}^{h'}(c_j)) = P_{c,a}^h(c_i)$ for some $h"$ less than half the period of the cycle attracting $c_i$ and $c_j$. 
Summarizing, we may assume that $l\le n_{k,i}$ and $l' \le n_{k,j} + \frac{m_{k,j} -n_{k,j}}2$.

\begin{lemma}\label{lmdegenerate}
Pick any two sequences of $(d-1)$-tuples $(\um_k,\un_k)$ of non-negative integers such that 
$m_{k,i}> n_{k,i}$ and $m_{k,j}\rightarrow\infty$ for all $0\le i\le d-2$. 
\begin{center}
$ \card \left( \{ (c,a)\in \PCF(\um_k,\un_k), \, \text{ some critical point is degenerate} \} \right) = o (d^{|\um_k|})~.$
\end{center}
\end{lemma}
We may thus restrict our attention to parameters having only simple critical points. 

As in the previous Section~\cite[Theorem 7.18]{favredujardin} applies. The subset of $\PCF^*(\um_k,\un_k)$ with only simple critical points, and such that $P^l_{c,a}(c_i) =P^{l'}_{c,a}( c_j)$ is in bijection with a subset of
\begin{multline*}
\mathsf{C}_k(i,j,l,l')
:=  \{  \Theta \in \mathsf{Cb}_0 ,\, \theta_h = (\a_h, \a'_h),\, 
d^{m_{k,h}} \a_h = d^{n_{k,h}} \a_h, \text{ for all } h\neq j \\
\text{ and }
d^{l'} \theta_j = d^l \theta_i
\}~.
\end{multline*}
For any $(i,j,l,l')$, the set $\mathsf{C}_k(i,j,l,l')$ has cardinality at most
\begin{eqnarray}
d^{d-1} \, d^{l'}\, \prod_{h\neq j}d^{m_{k,h}-n_{k,h}}  \le d^{|\um_k| - m_{k,j} + l' + d-1}~.
\label{ineg:victory}
\end{eqnarray}
Denote by $C(k)$ the right hand side of~\eqref{eqwouf}.
Then our discussion shows that
\begin{eqnarray*}
C(k) & \le & \sum_{i\neq j} \sum_{l \le n_k,j} \sum_{l' \le n_{k,j} +\frac12 (m_{k,j} - n_{k,j})} \card (\mathsf{C}_k(i,j,l,l'))\\
 & \le &
d^{d-1} \, d^{|\um_k|} \left(  \sum_{i \neq j } n_{k,i} m_{k,j}\, d^{- \frac{m_{k,j} - n_{k,j}}{2}} \right)= o (d^{|\um_k|} )
\end{eqnarray*}
since by assumption $m_{k,i} - n_{k,i} \to \infty$.
\end{proof}

\begin{proof}[Proof of Lemma~\ref{lmdegenerate}]
A critical point $c_i$ of $P_{c,a}$ is degenerate when $\ord_{c_i}(P_{c,a}) \ge 3$.
This is equivalent to having $c_i = c_j$ for some $i\neq j$, whence
\begin{eqnarray*}
\{ (c,a)\in \PCF(\um_k,\un_k), \, \text{ some critical point is degenerate} \}=\bigcup_{i\neq j}\{c_i=c_j\}\cap \PCF(\um_k,\un_k)~.
\end{eqnarray*}
Since $\PCF(\um_k,\un_k)\cap\{c_i=c_j\}=\{c_i=c_j\}\cap\bigcap_{h\neq i}\Per_h(m_{k,h},n_{k,h})$,  Bezout's theorem implies
\begin{eqnarray*}
\card(\PCF(\um_k,\un_k)\cap\{c_i=c_j\})\leq d^{|\um_k|-\um_{k,i}}\leq d^{|\um_k|-\min_{0\leq l\leq d-2}m_{k,l}}
\end{eqnarray*}
so that 
$$
\card \left( \{ (c,a)\in \PCF(\um_k,\un_k), \, \text{ some critical point is degenerate} \} \right)
\le C\, d^{|\um_k| - \min_i m_{k,i}}~,
$$
with $C=\card(\{ (i,j), \ i\neq j\}$.
Since $m_{k,i} \to \infty$ for all $i$, the result follows.
\end{proof}


\subsection{Proof of Theorem~\ref{maintm1}.} 
We rely on the following key
\begin{lemma}\label{lemSPCFcount}
Pick any two sequences of $(d-1)$-tuples $(\um_k,\un_k)$ of non-negative integers such that 
$m_{k,i}> n_{k,i}$ and $m_{k,i}-n_{k,i} \rightarrow\infty$ for all $0\le i\le d-2$. For any proper algebraic subvariety $V\subset\A^{d-1}_\C$, we have
$$ 
\lim_{k\to\infty}
\frac{\card \left( V\cap \PCF^*(\um_k,\un_k) \right)}{\card (\PCF^*(\um_k,\un_k) )} = 0~.$$
\end{lemma}
Let $\um_k = (m_{k,0}, \ldots, m_{k,d-2})$ and
$\un_k = (n_{k,0}, \ldots, n_{k,d-2})$. We want to apply Theorem~\ref{tmyuan} to
$Z_k:= \PCF^* (\um_k,\un_k)$ and the metrics induced by 
$G_v$ on $\O(1) \to \p^{d-1}_{\C_v}$ for each place $v \in M_\Q$.
By Theorem~\ref{thm:adelic}, these metrics induce a semipositive adelic metric.

Note that a postcritically finite polynomial that is not strictly postcritically finite
admits a periodic critical point. It follows that $\PCF^* (\um_k,\un_k)$ are all defined by
equations defined over $\Q$.
It is also clear that $h_\bif(Z_k) =0$ for each $k$. It is however not true that $Z_k$ is 
generic. To get around this problem, we proceed as follows. 

First we enumerate all irreducible hypersurfaces $\{ D_q\} _{q \in \N}$ of 
$\A_\C^{d-1}$ that are defined over $\Q$. Fix $\e>0$. We shall construct a 
sequence of sets $Z'_{k,\e} \subset Z_k$ such that:
\begin{enumerate}
\item
$\card (Z'_{k,\e}) \ge (1-\e) \card (Z_k)$ for all $k$;
\item
$Z'_{k,\e}$ is $\textup{Gal} (\bar{\Q}/\Q)$-invariant;
\item
for any $q$, $Z'_{k,\e} \cap D_q = \emptyset$ for all $k\ge K(q)$ large enough.
\end{enumerate}
Suppose for a moment that we have found such a sequence. The last condition 
implies $Z'_{k,\e}$ to be generic. By Theorem~\ref{tmyuan} we conclude that 
$$
\mu_{k,\e}'\pe\frac1{\card (Z'_{k,\e})} \sum_{x\in Z_{k,\e}'}\delta_x \to (dd^cG)^{d-1}=\mu_\bif~.
$$
Now pick any continuous function $\varphi$ with compact support on $\C^{d-1}$. Then 
\begin{eqnarray*}
\left|\int \varphi\, \mu_k - \int \varphi \, \mu_\bif\right|
\le 
\left|\int \varphi\,  \mu_k -  \int \varphi\,\mu_{k,\e}'\right| +
\left| \int \varphi\,\mu_{k,\e}'- \int \varphi \, \mu_\bif\right|.
\end{eqnarray*}
The second term in the sum tends to $0$ as $k\to \infty$. The first one can be bounded from above as follows:
\begin{eqnarray*}
\left|\int \varphi\,  \mu_k -  \int \varphi\,\mu_{k,\e}'\right| & \le & 
\frac1{\card (Z_k)}\int |\varphi|\sum_{x\in Z_k\setminus Z_{k,\e}'}\delta_x\\
& & + \left( \frac1{\card (Z'_{k,\e})} - \frac1{\card (Z_k)} \right)\int|\varphi|\sum_{x\in Z_{k,\e}'}\delta_x\\
& \le & 2\,\e\sup |\varphi|~.
\end{eqnarray*}
This shows that 
$$\limsup_{k\to\infty} \left|\int \varphi\,  \mu_k - \int \varphi \, \mu_\bif\right| \le 2\,\e \sup |\varphi|$$
for all $\e>0$. Letting $\e\to0$ we get that $\int \varphi\,  \mu_k \to \int \varphi \mu_\bif$,
and this concludes the proof.

\medskip

We are thus left with the construction of the sequence $Z'_{k,\e}$.
To do so we proceed as follows. 
By Lemma~\ref{lemSPCFcount}, for all $k\ge k_1$ we have
$\card (D_1\cap Z_k) \le \frac{\e}{4} \card (Z_k)$.
We set $Z^{(1)}_k = Z_k$ if $k < k_1$, and  $Z^{(1)}_k := Z_k \setminus (Z_k\cap D_1)$ if $k\ge k_1$.
Observe that by construction $\card (Z^{(1)}_k) \ge (1 - \frac{\e}{4}) \card (Z_k)$ for all $k$, and
$Z^{(1)}_k$ are all $\textup{Gal} (\bar{\Q}/\Q)$-invariant since $D_1$ is defined over $\Q$.

Next we find $k_2 > k_1$ such that $\card (D_2\cap Z_k) \le \frac{\e}{8} \card (Z_k)$ for all $k\ge k_2$.
And we set  $Z^{(2)}_k = Z^{(1)}_k$ if $k < k_2$, and  $Z^{(2)}_k := Z^{(1)}_k \setminus (Z^{(1)}_k\cap D_2)$ if $k\ge k_2$. Here again $Z^{(2)}_k$ are all $\textup{Gal} (\bar{\Q}/\Q)$-invariant, and
we have $\card (Z^{(2)}_k) \ge (1 - \frac{\e}{4})  (1 - \frac{\e}{8})\card (Z_k)$ for all $k$.

Recursively we find $k_j > k_{j-1}$ such that $\card (D_j\cap Z_k) \le \frac{\e}{2^{j+1}} \card (Z_k)$ for all $k\ge k_j$.
And we set  $Z^{(j)}_k = Z^{(j-1)}_k$ if $k < k_j$, and  $Z^{(j)}_k := Z^{(j-1)}_k \setminus (Z^{(j-1)}_k\cap D_j)$ if $k\ge k_j$. These are $\textup{Gal} (\bar{\Q}/\Q)$-invariant finite sets such that $\card (Z^{(j)}_k) \ge \prod_{1\le l \le j} (1 - \frac{\e}{2^{j+1}})\card (Z_k)$ for all $k$.

We set $Z'_{k,\e} := Z_k^{(j)}$ for all $k\le k_j$. This definition is coherent since $ Z_k^{(j)} =  Z_k^{(j')}$
for all $k < \min\{ k_j, k_{j'}\}$. The sets  $Z'_{k,\e}$ are $\textup{Gal} (\bar{\Q}/\Q)$-invariant
since all $ Z_k^{(j)}$ are. We have 
$$
\card (Z'_k) \ge \prod_{j\ge 1} \left(1 - \frac{\e}{2^{j+1}}\right)\, \card (Z_k)
\ge (1 -\e)\, \card (Z_k) $$ 
for all $k$. Finally, pick any integer $q\ge1$. Then for $k\ge k_q$ we have
$Z'_{k,\e} \cap D_q \subset Z_k^{(q)} \cap D_q = \emptyset$. This completes the construction of the sequence $Z'_{k,\e}$.

\begin{proof}[Proof of Lemma~\ref{lemSPCFcount}]
Theorem~\ref{thm:countSPCF} implies $\card (\PCF^*(\um_k,\un_k))\ge c\, d^{|\um_k|}$ for some positive $c>0$.
On the other hand, Theorem~\ref{thmupperbound} implies
$\card \left( V\cap \TPCF(\um_k,\un_k) \right) = o(d^{|\um_k|})$.
By Theorem~\ref{cortransverse}, the complement of $\TPCF(\um_k,\un_k)$ in $\PCF^*(\um_k,\un_k)$ is included in the set where a critical relation appears, whose cardinality is $o( d^{|\um_k|})$ by Theorem~\ref{thmcriticalrelations}.
Whence
\begin{multline*}
\card \left( V\cap \PCF^*(\um_k,\un_k) \right) \le \card \left( V\cap \TPCF(\um_k,\un_k) \right) \\ + 
\card \left( \PCF^*(\um_k,\un_k)\setminus \TPCF(\um_k,\un_k) \right) = o( d^{|\um_k|})~,
\end{multline*}
which concludes the proof.
\end{proof}

\begin{remark}
One can ask whether the assumptions $m_{k,i}> n_{k,i}>0$ and $m_{k,i} - n_{k,i} \rightarrow \infty$ can be weakened to $m_{k,i} \ge n_{k,i}\ge 0$ and $m_{k,i} \to \infty$. 
Removing these assumptions would probably require the notion of Hubbard trees which classify postcritically finite polynomials up to conjugacy, see~\cite{Poirier}. 
\end{remark}


\section{Distribution of hyperbolic parameters.}
We aim at proving Theorem~\ref{maintm3} from the introduction.
To that end, we first use Yuan's result to prove Theorem~\ref{maincrittm3}
from which it is not difficult to deduce Theorem~\ref{maintm3}
in the case all multipliers are equal to $0$.
Then we extend it to arbitrary multipliers of norm $<1$ using a paramaterization of hyperbolic components of the interior of the connectedness locus by the multipliers of the attracting cycles and Briend-Duval's length-area estimate for holomorphic disks (see \cite[Appendix]{briendduval2}).

\subsection{Equidistribution of centers of hyperbolic components.}~
To simplify notation, we denote by $\Per_i^*(m)$ (resp. $\Per_i(m)$) the set of polynomials $P_{c,a}$ such that 
$c_i$ is a periodic point of period exactly (resp. divisible by) $m$. 
Finally for any $(d-1)$-tuple $\um=(m_0,\ldots,m_{d-2})$, we also  set
$$
\Per(\um)\pe\bigcap_{i=0}^{d-2}  \Per_i(m_i)  \text{ and }
\Per^*(\um)\pe\bigcap_{i=0}^{d-2}  \Per_i^*(m_i)  ~.$$
We begin with proving Theorem~\ref{maincrittm3} from the introduction. 
We follow exactly the same approach as for proving Theorem~\ref{maintm1}.

First we have the following transversality result. 
\begin{theorem}\label{cortransverse2}
Let $\um$ be any $(d-1)$-tuple of integers such that $m_j>0$ for all $0\le j \le d-2$. Let $(c,a)\in\Per(\um)$ be such that $P_{c,a}$ has only simple critical points. Then the $(d-1)$ hypersurfaces $\Per_j(m_j)$ are smooth and intersect transversely $(c,a)$.
\end{theorem}

\begin{proof}[Proof of Theorem~\ref{cortransverse2}]
The proof is similar to the one of Theorem~\ref{cortransverse}, replacing
 Corollary~\ref{corAdam} by
 Corollary~\ref{corAdam2}.
\end{proof}

Next we estimate the proportion of points of $\Per^*(\um_k)$ lying in a fixed proper subvariety.

\begin{theorem}\label{thm:ouf}
Pick any sequence of $(d-1)$-tuples $(\um_k)$ of non-negative integers such that 
$m_{k,j}\rightarrow\infty$ for all $0\le j\le d-2$ and $m_{k,i}\neq m_{k,j}$ for all $i\neq j$. For any proper algebraic subvariety $V\subset\A^{d-1}_\C$, we have
$$ 
\lim_{k\to\infty}
\frac{\card \left( V\cap \Per^*(\um_k) \right)}{\card (\Per^*(\um_k) )} = 0~.$$
\end{theorem}

\begin{proof}[Proof of Theorem~\ref{thm:ouf}] 
Since $m_{k,i} \neq m_{k,j}$ for all $i\neq  j$, any point in $\Per^*(\um_k)$ has $(d-1)$ critical points, and these critical points are necessarily simple. 
Theorem~\ref{cortransverse2} thus applies for each point in 
$\Per(\um_k)$. In particular, it applies to any point in $\Per^*(\um_k)$, and Lemmas~\ref{lmtransversal} and~\ref{lmdegree} show that 
$$
\card \left( V\cap \Per^*(\um_k) \right) \le
\sum_{|I| =q} \deg(V) d^{\sum_{i\in I} m_{k,i}}
$$
just as in the proof of Theorem~\ref{thmupperbound}.
To estimate  $\card (\Per^*(\um_k) )$ from below, we rely on 
\begin{lemma}
For all $m\ge1$, 
$$\deg(\Per^*_i(m))\ge  (1 - d^{-1}) \,d^m~.$$
\label{lm:lwerdeg}
\end{lemma}
By Theorem~\ref{cortransverse2}, the  $(d-1)$ hypersurfaces $\{ \Per^*_i(m_{k,i})\}_{0\le i \le d-2}$
intersect transversely at any point of the finite set $\Per^*(\um_k)$. According to Lemma \ref{lm:lwerdeg},
Bezout's Theorem gives 
$$\card(\Per^*(\um_k))\ge \prod_{i=0}^{d-2}\deg(\Per^*_i(m_{k,i}))\ge(1-d^{-1})^{d-1}d^{|\um_k|}~,$$
which ends the proof.
\end{proof}

\begin{proof}[Proof of Lemma~\ref{lm:lwerdeg}]
Since $\Per_i(m) = \sum_{l | m} \Per_i^*(l)$, the M\"obius inversion formula implies
$$
\deg(\Per_i^*(m) )= \sum_{l | m} \mu\left(\frac{m}{l}\right) \Per_i(l)~.
$$
It follows that
$$
\deg(\Per_i^*(m) ) \ge d^m - \sum_{l \le m/2} d^l \ge (1 - d^{-1}) \,d^m  ~,
$$
as required.
\end{proof}


\begin{proof}[Proof of Theorem \ref{maincrittm3}]
Observe that each hypersurface $\Per^*_i(m_{k,i})$ is defined over $\Q$, since $\Per_i(m) = \sum_{l | m} \Per_i^*(l)$. It follows that $\Per^*(\um_k)$ is $\textup{Gal}(\bar{\Q}/\Q)$-invariant. For any point $(c,a)$ in this set the critical points have a finite orbit, hence $G(c,a) =0$ and $h_\bif(c,a) =0$.
We may thus apply Theorem~\ref{tmyuan} to the adelic metrized bundle given by Theorem~\ref{thm:adelic}
and the set $Z'_k := \Per^*(\um_k)$.
This  sequence of finite sets is not generic but Theorem~\ref{thm:ouf} allows one to copy the proof we used 
for Theorem~\ref{maintm1} to conclude. 
\end{proof}

As a corollary, we can prove Theorem \ref{maintm3} in the case when $w_0=\cdots=w_{d-2}=0$.

\begin{corollary}\label{corcvPern}
For all $0\leq i\leq d-2$,  choose a sequence of integers $m_{k,i}$ s.t. $m_{k,i} \rightarrow \infty$ as $k\rightarrow\infty$.  Assume in addition that $m_{k,i}\neq m_{k,j}$ for all $k$ and all $i\neq j$.
Consider the probability measure $\mu''_k$
 that is uniformly distributed over all parameters $(c,a)\in \C^{d-1}$ 
admitting $(d-1)$ super-attracting periodic orbits of length $m_{k,1}, \ldots , m_{k,d-1}$ respectively.

Then the measures $\mu''_k$ converge in the weak sense to $\mu_\bif$ as $k\rightarrow\infty$.
\end{corollary}

\begin{proof}
For any permutation 
$\sigma\in\mathfrak{S}_{d-1}$, denote by $\mu_{\sigma,k}$ the probability measure equidistributed on $\bigcap_{j=0}^{d-2}\Per^*_{\sigma(j)}(m_{k,j})$.
We observe that the support of these measures are  disjoint for any two distinct permutations, and that 
$(d-1)!\, \mu''_k = \sum_{\sigma\in\mathfrak{S}_{d-1}}\mu_{\sigma,k}$. By Theorem~\ref{maincrittm3}, $\mu_{\sigma,k}\to \mu_\bif$ for any $\sigma$, hence $\mu''_k\to \mu_\bif$.
\end{proof}

\begin{remark}
Theorem~\ref{cortransverse2} shows that the cardinality of the support of $\mu''_k$ is equivalent to $(d-1)!\, d^{|\um_k|}$ when $k\to\infty$. 
\end{remark}

\subsection{The algebraic varieties $\Per^*(n,w)$.}\label{section:Per}

\par In this section, we explain how to parameterize the set of polynomials $P\in\poly_d$ possessing a cycle with a given multiplier and period following~\cite{Silverman} and \cite[\S 2.1]{BB2}.

\begin{theorem}
For any $n\geq1$ there exists a polynomial function $q_n:\poly_d\times\C\longrightarrow\C$ such that :
\begin{enumerate}
	\item For any $w\in \C\setminus\{1\}$, $q_n(P,w)=0$ if and only if $P$ has a cycle of exact period $n$ and of multiplier $w$;
	\item $q_n(P,1)=0$ if and only if $P$ has a cycle of period $n$ and multiplier $1$ or $P$ has a cycle of period $m$ and multiplier a $r$-th primitive root of unity with $n=mr$.
\end{enumerate}
\label{tmdefPern}
\end{theorem}
\begin{proof}[Sketch of proof]
Define
$$
\Phi_n (P,z) \pe P^n(z) -z, \text{ and } \Phi^*_n (P,z) \pe \prod_{l|n} \Phi_l(P,z)^{\mu\left( \frac{n}{l}\right)}
~.$$
Then for all $P\in \poly_d$,  the roots of $\Phi_n^*(P,\cdot)$ are either simple roots at the $n$-periodic points of $P$, or  multiple roots at the periodic points of $P$ with exact period $m$ dividing $n$ and multiplier $w$ satisfying $w^r=1$ for $r=n/m\ge2$.

According to \cite[Theorem 2.3.4 and Proposition 2.3.5]{bsurvey}, see also~\cite[Chapter~4]{Silverman}, $\Phi^*_n$ is a polynomial function on $\poly_d \times \C$ such that $$\nu_d(n)\pe\deg_z(\Phi_n^*)\sim d^n, \text{ and } \mu_d(n)\pe\deg_P(\Phi_n^*)\sim (d-1)^{-1} d^n~.$$
The projection map $\pi: (\Phi_n^*)^{-1} (0) \to \poly_d$ is a proper ramified cover of degree $\nu_d(n)$. 
For any polynomial function $H: \poly_d \times \C \to \C$, and for any symmetric function $\sigma_i$ on 
$\nu_d(n)$ symbols, the following function
$T_i(H): \poly_d \to \C$ defined by $$T_i(H)(P) \pe \sigma_i ( H(P_{1}) , \ldots, H(P_{\nu_d(n)}))$$
is polynomial
where $\pi^{-1}(P) = \{ H(P_{1}) , \ldots, H(P_{\nu_d(n)}) \}$ possibly with repetitions.
This follows for instance from the next lemma, whose proof is left to the reader.
\begin{lemma}\label{lmresul}
Let $A(T) = a_0 + a_1 T + \ldots + a_k T^k$ and $B(T) = b_0 + b_1 T + \ldots + b_l T^l$ 
be two complex polynomials of degree $k$ and $l$ respectively. 
Denote by $\alpha_1, \ldots , \alpha_k$ the roots of $A$ possibly with repetitions. For any symmetric 
function $\sigma_i$ of degree $i$ on $k$ symbols, one can write
$\sigma_i (B(\alpha_1), \ldots , B(\alpha_k))$ as a polynomial in the coefficients of $B$ and in the $a_j/a_k$'s.
\end{lemma}
Define
$$
r_n(P,w)\pe \sum_{i=0}^{\nu_d(n)} T_i((P^n)')(-w)^{\nu_d(n)-i}~.
$$
For a fixed $w\in \C$, we have $r_n(P,w) =0$  if and only if there exists a point $(P,z) \in \{ \Phi_n^* = 0\}$ such that $(P^n)'(z) =w$. 
Since the  multiplier is constant on any point in the same periodic orbit,
it follows that for any fixed $P$, the polynomial $r_n(P,w)$ has all its root of multiplicity $n$.
Whence there exists a unique polynomial function $q_n: \poly_d \times \C \to \C$ such that
$$
q_n(P,w) ^n = r_n(P,w)~.
$$
Properties (1) and (2) now follow  from the definition.
\end{proof}

By setting $p_n(c,a,w)\pe q_n(P_{c,a},w)$ and using Lemma \ref{lm:lwerdeg}, we get
\begin{corollary}
For any $n\geq1$ there exists a polynomial function $p_n:\C^{d-1}\times\C\longrightarrow\C$ such that :
\begin{enumerate}
	\item For any $w\in \C\setminus\{1\}$, $p_n(a,c,w)=0$ if and only if $P_{c,a}$ has a cycle of exact period $n$ and of multiplier $w$;
	\item $p_n(a,c,1)=0$ if and only if $P_{c,a}$ has a cycle of period $n$ and multiplier $1$ or $P_{c,a}$ has a cycle of period $m$ and multiplier a $r$-th primitive root of unity with $n=mr$;
	\item $\bar{M}_d(n)\pe\deg_{(c,a)}p_n(\cdot,0)\ge d^{-1}(d-1)^2\, d^n$. 
\end{enumerate}
\label{cordefPern}
\end{corollary}
\begin{proof}
It only remains to prove $(3)$. Since $\{p_n(\cdot,0)=0\}=\bigcup_j\Per^*_j(n)$ and since, by Lemma \ref{lminfinity} and  Proposition \ref{propBH}, the $\Per^*_j(n)$ intersect properly, we see that $\deg(p_n(\cdot,0))\ge \sum_j\deg(\Per^*_j(n))$ and Lemma \ref{lm:lwerdeg} ends the proof.
\end{proof}

For $n\geq1$ and $w\in\C$ we set
$$\Per^*(n,w)\pe\{(c,a)\in\C^{d-1} , \, p_n(c,a,w)=0\}~.$$

When $\uw\pe(w_0,\ldots,w_{d-2})\in\D^{d-1}$ and $\um=(m_0,\ldots,m_{d-2})$ satisfy $m_j\neq m_i$ for all $i\neq j$, any parameter in the intersection of the hypersurfaces $\Per^*(m_j,w_j)$ has all its critical points in the filled-in Julia set. The intersection $\bigcap_{0\le j\leq d-2}\Per^*(m_j,w_j)$ is thus a compact algebraic set, i.e. a finite set we denote by
$$
\Per^*(\um,\uw)\pe\bigcap_{j=0}^{d-2}\Per^*(m_j,w_j)~.
$$
\begin{remark}
Observe that  $\Per^*(\um) \neq \Per^*(\um,\underline{0})$. The set $\Per^*(\um)$ is the set of polynomials for which the critical point $c_i$ is periodic of period exactly $m_i$ for each $0\le i \le d-2$, whereas $\Per^*(\um,\underline{0})$ is the set of polynomials for which for each $0\le i \le d-2$ there exists a critical point $c_{j_i}$ that is periodic of period exactly equal to $m_i$.
\end{remark}


\subsection{Parameterizing hyperbolic components with $(d-1)$ attracting cycles.}

Pick $n_0,\ldots,n_{d-2} \in \N^*$ such that $n_i\neq n_j$ if $i\neq j$, and consider a connected component $\mathcal{H}\subset\C^{d-1}$ of the set of $(c,a) \in \C^{d-1}$ such that $P_{c,a}$ admits $(d-1)$ distinct attracting periodic orbits of  exact periods $n_0, \ldots , n_{d-2}$ respectively.
Recall that $\mathcal{H}$ is open. 

Observe that all critical points of $P_{c,a} \in \mathcal{H}$ are attracted to one and only one attracting orbit, and thus are simple. For each $i$, we let $w_i(c,a) \in \D$ be the multiplier of the attracting periodic orbit that 
has exact period $n_i$. In this way we get a holomorphic map
$$\mathcal{W}(c,a)\pe (w_0(c,a),\ldots, w_{d-2}(c,a))~.$$

Following closely\cite[\S~2]{BB3}, we shall prove 
\begin{theorem}
The map $\mathcal{W}:\mathcal{H}\longrightarrow\D^{d-1}$ is a biholomorphism.
\label{tm:parametrize}
\end{theorem}

\begin{proof}
Since $\D^{d-1}$ is simply connected it is sufficient to prove that $\mathcal{W}$ is proper and locally invertible.

We first prove that $\mathcal{W}$ is proper by contradiction. Suppose there exists a sequence $(c_k,a_k)\in\mathcal{H}$ converging to $\partial\mathcal{H}$ such that 
$$\mathcal{W}(c_k,a_k)=(w_0(c_k,a_k),\ldots,w_{d-2}(c_k,a_k))\to (w_{0,\infty},\ldots,w_{d-2,\infty})\in\D^{d-1}~.$$
 Since $\mathcal{H}$ is bounded, we may assume that $(c_k,a_k)\to(c_\infty,a_\infty)$. 
\begin{lemma}
Suppose $P(z) = w z + \sum_{2 \le i \le D} a_i z^i$ is a polynomial map fixing the origin with $|w| <1$. Then for any $ r\le \frac{ |w|^{1/2} - |w|}{D \max \{ |a_i|\}}$, one has
$$
P (\D(0,r)) \subset \D( 0, \sqrt{|w|}\, r)~.
$$
\end{lemma}
\begin{proof}
If $M = \max \{ |a_i|\}$, then $|P(z) - wz | \le D\, M |z|^2$ for any $|z| \le 1$.
For any $|z| \le r$, we get
$|P(z) | \le |w| r + D\, M r^2\le |w|^{1/2} r$
as soon as $|w| + DM r \le |w|^{1/2}$.
\end{proof}
For each $0\le i \le d-2$,  pick $x_{i,k}$ a point in the attracting periodic orbit for $P_{c_k,a_k}$ which period is $n_i$. For each $k$, the orbit of $x_{i,k}$ attracts a unique critical point $c_{i,k}\pe c_{j_{i,k}}(P_{c_k,a_k})$.  Extracting further, we may assume that $j_{i,k}\equiv j_i$ doesn't depend on $k$. Since the period is fixed equal to $n_i$ for all $k$, the preceding lemma implies the existence of a fixed radius $r >0$ and a fixed positive constant $\e>0$ such that $P_{c_k,a_k}^{n_i}( \D(x_{i,k},r)) \subset  \D(x_{i,k}, (1-\e) r)$ for all $k$.
Since the cycles $\{P_{c_k,a_k}^m(x_{i,k})\}_{m\ge 0}$ and $\{P_{c_k,a_k}^m(x_{j,k})\}_{m\ge 0}$ are disjoint, it follows that 
$\min_{m,m'\ge 0, i\neq j} |P^m_{c_k,a_k}(x_{i,k}) - P^{m'}_{c_k,a_k}(x_{j,k})| \ge r$.

Extracting further if necessary we may also assume that $x_{i,k}$ is converging to a  periodic point $x_{i,\infty}$ of $P_{c_\infty,a_\infty}$. The previous estimate shows that $x_{i,\infty}$ is attracting, and further that  these $(d-1)$ cycles are  distinct. We conclude that $P_{c_\infty,a_\infty}$ is hyperbolic. Since the space of hyperbolic maps is open,
this would imply $P_{c_\infty,a_\infty}\in \mathcal{H}$,  which is a contradiction.

\medskip

Next we show that $\mathcal{W}$ is locally invertible.
Choose any base point $(c_0,a_0 ) \in \mathcal{H}$, and
pick any $\e>0$ small enough such that $\mathcal{W}(c_0,a_0) = (w_0, \ldots , w_{d-2})$ lies in the open polydisk of center $0$ and radius $1 -\e$. We shall first construct a continuous map $$\sigma : \D (0, 1 - \e) ^{d-1} \to \mathcal{H}$$ such that $ \mathcal{W} \circ \sigma = \id$ using quasi-conformal surgery. We only sketch the construction referring to~\cite[Theorem~VIII.2.1]{carleson} for detail.

The polynomial  $P_{c_0,a_0}$ has $(d-1)$ distinct attracting cycles. 
We let $U_{1,i}, \ldots , U_{n_i,i}$ be the immediate basin of attraction of these cycles
indexed by $1\le i \le d-1$ such that the unique critical point $c_{j_i}$
attracted to this cycle belongs to $U_{1,i}$.
Since it is a simple critical point, 
there exists a conformal map $\varphi : U_{1,i} \to \D$
such that 
$$\varphi \circ  P^{n_i} \circ \varphi^{-1} (\zeta) = \zeta\, \frac{\zeta +w_i }{ 1 + \bar{w_i} \zeta}, \text{ for any } |\zeta| < 1~.$$
For any $\lambda = (\lambda_0, \ldots , \lambda_{d-2}) \in \D (0, 1 - \e) ^{d-1}$, we construct a smooth
map $\tilde{P}_\lambda$ by setting 
$\tilde{P}_\lambda\pe P_{c_0,a_0}$
outside the union of all $U_{j,i}$, 
and such that 
$\varphi \circ  P^{n_i}_\lambda \circ \varphi^{-1} (\zeta) = \zeta\, \frac{\zeta +\lambda_i }{ 1 + \bar{\lambda_i} \zeta}$ on a fixed disk $|\zeta| < 1 -r$ containing the critical point of  the latter Blashke product. Details of the construction can be found in op. cit. 
In particular, one can see that $\tilde{P}_\lambda$ depends continuously on the parameter. 

We now solve the Beltrami equation for the unique Beltrami form which is $0$ on the complement  the $U_{j,i}$'s and invariant under $\tilde{P}_\lambda$. In this way we get a quasiconformal homeo\-morphism $\psi_\lambda : \C \to \C$ such that $\psi_\lambda (z) = z + o(1)$ at infinity, and $P_\lambda \pe \psi_\lambda \circ \tilde{P}_\lambda \circ \psi_\lambda^{-1}$ is holomorphic. At infinity, we see that 
$P_\lambda (z) = \frac1d z^d + O(z^{d-1})$. 

At this point we have constructed a continuous map
$\D (0, 1 - \e) ^{d-1} \to \poly_d$, $\lambda \mapsto P_\lambda$ such that $P_{\uw} = P_{c_0,a_0}$
and $P_\lambda$ admits $(d-1)$ attracting periodic cycles of exact period $m_i$ and multiplier $\lambda_i$ respectively. 

\medskip

Let us now prove that $\mathcal{W}$ is locally invertible in a neighborhood of $(c_0,a_0)$. 
Since $0$ is a simple critical point, we may find a holomorphic map $c : U \to \C$
defined on an open set $U\subset \poly_d$ containing $(c_0,a_0)$ and satisfying
$c( c_0,a_0) = 0$ and $P'(c(P)) =0$ for all $P\in U$.
We then set $\sigma(\lambda) \pe P_\lambda ( \cdot + c(P_\lambda) ) - c(P_\lambda)$.
In this way we get a polynomial of degree $d$ with dominant term equal to $\frac1d$ and
having $0$ as a critical point. Since any such polynomial is equal to some $P_{c,a}$ for a unique  $(c,a) \in \C^{d-1}$ we thus get a continuous map defined in a neighborhood of $(c_0,a_0)$ and such that $\mathcal{W} \circ \sigma = \id$.
The next lemma applied to $\phi \pe \mathcal{W}$ and $\varphi\pe \sigma$ implies $\mathcal{W}$ to be locally invertible at $(c_0,a_0)$ as required.
\end{proof}

\begin{lemma}\label{lm:genhol}
Let $\phi : (\C^n,0)  \to (\C^n,0)$ be a germ of holomorphic  map such that 
$\phi^{-1}(0) = (0)$. Suppose that there exists a germ of 
continuous map $\varphi:  (\C^n,0)  \to (\C^n,0)$ such that $\phi \circ \varphi = \id$.

Then  $\varphi$ is holomorphic and $\phi$ is locally invertible at $0$.
\end{lemma}

\begin{proof}[Proof of Lemma~\ref{lm:genhol}]
Since $\phi^{-1}(0) = (0)$ there exist open sets $U,V$ containing $0$ such that 
$\phi : U \to V$ is a finite branched cover, and $\varphi$ is defined over $V$. 
The critical values of $\phi$ define a hypersurface $H \subset V$. The holomorphic
inverse mapping theorem shows that $\varphi$ is holomorphic at any point in $V\setminus H$.
Since it is continuous, it extends holomorphically through $H$. The differential of $\phi$ and $\varphi$ are thus both invertible at $0$ and the result follows.
\end{proof}

\begin{remark}
We could also have used transversality arguments of Epstein \cite{epstein2} to get the local invertibility of the map $\mathcal{W}$ (see also Levin~\cite{Levin2}).
This alternative approach actually proves that 
the map $\mathcal{W}$ extends as a homeomorphism $\overline{\mathcal{W}}:\overline{\mathcal{H}}\longrightarrow\overline{\D}^{d-1}$. 
\end{remark}

\subsection{Proof of Theorem~\ref{maintm3}.}

Let now pick any $\uw\pe(w_0,\ldots,w_{d-2})\in\D^{d-1}$ and any sequence $\um_k$ of $(d-1)$-tuples with $m_{k,i}\to\infty$ for all $i$ and $m_{k,i}\neq m_{k,j}$ for all $i\neq j$. We want to prove
$ \mu''_k \longrightarrow\mu_\bif$,
where $\mu''_k$ is the probability measure that is uniformly distributed over the set $\Per^*(\um_k,\uw)$.
We write $\uw[0]\pe (0,\ldots,0)$ and for any $1\le j\le d-1$, we set
$$\uw[j]\pe (w_0,\ldots,w_{j-1},0,\ldots,0)~.$$
We denote by $\mu_{k,j}$ the measure that is uniformly distributed on $\Per^*(\um_k,\uw[j])$.
By Corollary \ref{corcvPern}, we know that $\mu_{k,0}\to\mu_\bif$. 
To conclude it is thus sufficient to  prove that for any $0\le j\le d-2$, we have  
\begin{eqnarray}
\mu_{k,j+1}-\mu_{k,j}\longrightarrow 0~.
\label{cvdiff}
\end{eqnarray}
Let us now fix $0\le j\le d-2$. If $w_j=0$, we have $\mu_{k,j+1}=\mu_{k,j}$ and the proof is finished. We thus assume that $w_j\in\D\setminus\{0\}$. For any $k$, we consider the algebraic subvariety 
$$C_{k,j} \pe\bigcap_{h< j}\Per^*(m_{k,h},w_h)\cap\bigcap_{l>j}\Per^*(m_{k,l},0)~.$$
Observe that  $C_{k,j} \cap  \Per^*(m_{k,j},0)$ is finite, hence $C_{k,j}$ is an algebraic curve.
Observe also that the $(d-1)$ hypersurfaces $\Per^*(m_{k,h},w_h)$ for $0\le h \le j-1$ (resp. $0\le h\le j$) and $\Per^*(m_{k,i},0)$ otherwise intersect transversally. Indeed any point in the intersection belongs to a hyperbolic component $\mathcal{H}$ for which Theorem~\ref{tm:parametrize} applies. The transversality statement then follows since the images of the hypersurfaces $\Per^*(m_{k,h},w_h)$ and  $\Per^*(m_{k,i},0)$ under $\mathcal{W}$ are coordinate hyperplanes.

Pick any point $(c,a) \in \Per^*(\um_k,\uw[j])$, and let $\mathcal{H}$ be the hyperbolic component containing $(c,a)$. Using Theorem~\ref{tm:parametrize}, we define $\phi_{c,a} : \D(0, |w_j|^{-1/2}) \to \mathcal{H}$ by setting 
$$\phi_{c,a}(t) \pe \mathcal{W}^{-1} (w_0, \ldots , w_{j-1}, t w_j , 0, \ldots 0) ~.$$
By construction, the disk $\D_{c,a} \pe \phi_{c,a}( \D(0, |w_j|^{-1/2}))$ is included in $\mathcal{H} \cap C_{k,j}$, 
$\phi_{c,a}(0) = (c,a)$ and $\phi_{c,a} (1)$ belongs to $\Per^*(\um_k,\uw[j+1])$.
Any hyperbolic component contains a unique point in $\Per^*(\um_k,\uw[j])$, hence
the collection of disks $\D_{c,a}$ is disjoint.
Note also  that  any point in $\Per^*(\um_k,\uw[j+1])$ belongs to a hyperbolic component, and thus
is equal to  $\phi_{c,a} (1)$ for a unique $(c,a) \in \Per^*(\um_k,\uw[j])$.

~

We conclude from these two discussions that one can count precisely the cardinality of the set $\Per^*(\um_k,\uw[j])$ for all $j$. First remark that we have proved
$$\card(\Per^*(\um_k,\uw[j+1]))=\card(\Per^*(\um_k,\uw[j])).$$
Using Bezout's theorem, Corollary~\ref{cordefPern} and an induction on $j$ we find
$$\card(\Per^*(\um_k,\uw[j+1])) = \card(\Per^*(\um_k,\uw[0])) = \prod_{l}\bar{M}_d(m_{k,l})~.$$ 
Since $\deg(\Per^*(m_{k,j},0))=\bar{M}_d(m_{k,j})$ and $C_{k,j}\cap \Per^*(m_{k,j},0) = \Per^*(\um_k,\uw[j])$ this gives $\deg (C_{k,j}) = \prod_{l\neq j}\bar{M}_d(m_{k,l})$ from which we infer
\begin{equation}\label{eqBD}
\frac{\card(\Per^*(\um_k,\uw[j+1]))}{\deg (C_{k,j})} = \bar{M}_d(m_{k,j})^{-1} \le d(d-1)^{-2}\, d^{-m_{k,j}}\to 0~.
\end{equation}

We fix any k\"ahler form $\om$ on $\p^{d-1}$ normalized so that $\int\om^{d-1}=1$. Recall that the area of any holomorphic curve $C\subset \p^{d-1}$ is defined by 
$\area ( C) \pe \int_C \om$ which is equal to $\deg(C)$ when $C$ is algebraic. We thus have
\begin{eqnarray}
 \sum_{(c,a) \in \Per^*(\um_k,\uw[j]) }\area (\D_{c,a}) 
\le \deg (C_{k,j})~.
 \label{ineg:volume}
\end{eqnarray}
Choose $\varepsilon>0$  sufficiently small. 
Denote by $\mathcal{B}_k$ the set of $(c,a) \in \Per^*(\um_k,\uw[j])$ such that 
$\area ( \D_{c,a}) \le \e^2$. 
Then by~\eqref{eqBD} there exists $$ \card ( \mathcal{B}_k) \ge \card (\Per^*(\um_k,\uw[j])) - \frac{\deg (C_{k,j})}{\e^2} 
\ge ( 1- \e) \card (\Per^*(\um_k,\uw[j]))
$$
if $k$ is large enough.
We now rely on the following length-area estimate.
\begin{lemma}[\cite{briendduval2}]
There exists $c>0$, such that for any holomorphic disks $D_1\Subset D_2\subset\p^{d-1}_\C$, 
$$
\left(\textup{diam}(D_1)\right)^2\le c\cdot \frac{\area(D_2)}{\min(1,\textup{mod}(A))}~,
$$
where $A$ is the annulus $D_2\setminus D_1$.
\label{lm:BD}
\end{lemma}
For each disk $\D_{c,a}\in \mathcal{B}_k$ we conclude that 
$$d_\om ( \phi_{c,a}(1) , \phi_{c,a}(0) ) \le K \e$$ with
$K\pe  \sqrt{c} / \min \{1, \log |w_j|^{-1/4}\}$, where $d_\om$ denotes the distance computed with respect to the k\"ahler form $\om$.
We  conclude the proof in the following way. Let $\varphi:\C^{d-1}\longrightarrow\R$ be a smooth function with compact support and write  $N_k \pe \card ( \Per^*(\um_k,\uw[j]) ) = \card (\Per^*(\um_k,\uw[j+1]))$ to simplify notation. Then we have
\begin{eqnarray*}
\left| \int \varphi\, \mu_{k,j+1}-\int \varphi \,\mu_{k,j}  \right|
& = & 
\frac1{N_k}
\left| \sum_{\Per^*(\um_k,\uw[j+1])}\varphi-\sum_{\Per^*(\um_k,\uw[j])}\varphi\right|
\\
& = & 
\frac1{N_k}
\left| \sum_{(c,a) \in \Per^*(\um_k,\uw[j])}\varphi (\phi_{c,a}(1) )  - \varphi (\phi_{c,a}(0) )\right|
\\
& \le  & 
\frac1{N_k}
 \sum_{\mathcal{B}_k}\left|\varphi (\phi_{c,a}(1) )  - \varphi (\phi_{c,a}(0) )\right|
+ 
\frac1{N_k}
\,  (2 \sup |\varphi| \e N_k)
\\
& \le  & 
\frac1{N_k}
(|\varphi|_{\mathcal{C}^1}\e N_k)
+ 
2 \sup |\varphi| \e  \le 3\e |\varphi|_{\mathcal{C}^1}~.
\end{eqnarray*}
Since $\e$ can be chosen arbitrarily small, we have $\mu_{k,j+1}-\mu_{k,j}\to0$, as required.


\bibliographystyle{short}
\bibliography{biblio}

\begin{thebibliography}{GHT}

\bibitem[B]{bsurvey}
Fran\c{c}ois Berteloot.
\newblock Bifurcation currents in holomorphic families of rational maps, 2011.
\newblock CIME Lecture Notes Springer To appear.

\bibitem[BB1]{BB1}
Giovanni Bassanelli and Fran{\c{c}}ois Berteloot.
\newblock Bifurcation currents in holomorphic dynamics on {$\mathbb{P}^k$}.
\newblock {\em J. Reine Angew. Math.}, 608:201--235, 2007.

\bibitem[BB2]{BB3}
Giovanni Bassanelli and Fran{\c{c}}ois Berteloot.
\newblock Lyapunov exponents, bifurcation currents and laminations in
  bifurcation loci.
\newblock {\em Math. Ann.}, 345(1):1--23, 2009.

\bibitem[BB3]{BB2}
Giovanni Bassanelli and Fran{\c{c}}ois Berteloot.
\newblock Distribution of polynomials with cycles of a given multiplier.
\newblock {\em Nagoya Math. J.}, 201:23--43, 2011.

\bibitem[BD1]{BD}
Matthew Baker and Laura Demarco.
\newblock Special curves and postcritically-finite polynomials, 2012.
\newblock preprint arXiv : math.DS/1211.0255.

\bibitem[BD2]{briendduval2}
Jean-Yves Briend and Julien Duval.
\newblock Deux caract\'erisations de la mesure d'\'equilibre d'un endomorphisme
  de {${\rm P}^k(\bold C)$}.
\newblock {\em Publ. Math. Inst. Hautes \'Etudes Sci.}, (93):145--159, 2001.

\bibitem[BE]{buffepstein}
Xavier Buff and Adam~L. Epstein.
\newblock Bifurcation measure and postcritically finite rational maps.
\newblock In {\em Complex dynamics : families and friends / edited by Dierk
  Schleicher}, pages 491--512. A K Peters, Ltd., Wellesley, Massachussets,
  2009.

\bibitem[BFH]{BFH}
Ben Bielefeld, Yuval Fisher, and John Hubbard.
\newblock The classification of critically preperiodic polynomials as dynamical
  systems.
\newblock {\em J. Amer. Math. Soc.}, 5(4):721--762, 1992.

\bibitem[BG1]{Mod2}
Fran\c{c}ois Berteloot and Thomas Gauthier.
\newblock On the geometry of bifurcation currents for quadratic rational maps,
  2012.
\newblock preprint.

\bibitem[BG2]{Article2}
Xavier Buff and Thomas Gauthier.
\newblock Pertubations of flexible {L}att\`es maps.
\newblock to appear in \textit{Bull. Soc Math. France,} 2011.

\bibitem[BH1]{baker-hsia}
M.~Baker and L.~H. H'sia.
\newblock Canonical heights, transfinite diameters, and polynomial dynamics.
\newblock {\em J. Reine Angew. Math}, (585):61--92, 2005.

\bibitem[BH2]{BH}
Bodil Branner and John~H. Hubbard.
\newblock The iteration of cubic polynomials. {I}. {T}he global topology of
  parameter space.
\newblock {\em Acta Math.}, 160(3-4):143--206, 1988.

\bibitem[BT]{BedfordTaylor}
Eric Bedford and B.~A. Taylor.
\newblock The {D}irichlet problem for a complex {M}onge-{A}mpere equation.
\newblock {\em Bull. Amer. Math. Soc.}, 82(1):102--104, 1976.

\bibitem[CG]{carleson}
Lennart Carleson and Theodore~W. Gamelin.
\newblock {\em Complex dynamics}.
\newblock Universitext: Tracts in Mathematics. Springer-Verlag, New York, 1993.

\bibitem[CL1]{ACL}
Antoine Chambert-Loir.
\newblock Mesures et \'equidistribution sur les espaces de {B}erkovich.
\newblock {\em J. Reine Angew. Math.}, 595:215--235, 2006.

\bibitem[CL2]{ACL2}
Antoine Chambert-Loir.
\newblock Heights and measures on analytic spaces. {A} survey of recent
  results, and some remarks.
\newblock In {\em Motivic integration and its interactions with model theory
  and non-{A}rchimedean geometry. {V}olume {II}}, volume 384 of {\em London
  Math. Soc. Lecture Note Ser.}, pages 1--50. Cambridge Univ. Press, Cambridge,
  2011.

\bibitem[De]{DeMarco1}
Laura DeMarco.
\newblock Dynamics of rational maps: a current on the bifurcation locus.
\newblock {\em Math. Res. Lett.}, 8(1-2):57--66, 2001.

\bibitem[Du]{dsurvey}
Romain Dujardin.
\newblock Bifurcation currents and equidistribution on parameter space, 2011.
\newblock preprint.

\bibitem[DF]{favredujardin}
Romain Dujardin and Charles Favre.
\newblock Distribution of rational maps with a preperiodic critical point.
\newblock {\em Amer. J. Math.}, 130(4):979--1032, 2008.

\bibitem[E]{epstein2}
Adam Epstein.
\newblock Transversality in holomorphic dynamics, 2009.
\newblock preprint.

\bibitem[F]{Fisher}
Yuval Fisher.
\newblock The classification of critically preperiodic polynomials, 1989.
\newblock {P}h{D} {T}hesis, {C}ornell {U}niversity.

\bibitem[FRL]{FRL}
C.~Favre and J.~Rivera-Letelier.
\newblock Equidistribution quantitative des points de petite hauteur sur la
  droite projective.
\newblock {\em Math. Ann.}, 335(2):311--361, 2006.

\bibitem[G]{Article1}
Thomas Gauthier.
\newblock Strong bifurcation loci of full {H}ausdorff dimension.
\newblock {\em Ann. Sci. \'Ec. Norm. Sup\'er. (4)}, 45(6):947--984, 2012.

\bibitem[GHT]{GHT}
Dragos Ghioca, Liang-Chung H'sia, and Thomas Tucker.
\newblock Preperiodic points for families of rational maps.
\newblock Preprint, 2012.

\bibitem[HS]{Silvermandiophantine}
Marc Hindry and Joseph~H. Silverman.
\newblock {\em Diophantine geometry}, volume 201 of {\em Graduate Texts in
  Mathematics}.
\newblock Springer-Verlag, New York, 2000.
\newblock An introduction.

\bibitem[I]{Ingram}
Patrick Ingram.
\newblock A finiteness result for post-critically finite polynomials.
\newblock {\em Int. Math. Res. Not. IMRN}, (3):524--543, 2012.

\bibitem[K]{kiwi-portrait}
Jan Kiwi.
\newblock Combinatorial continuity in complex polynomial dynamics.
\newblock {\em Proc. London Math. Soc. (3)}, 91(1):215--248, 2005.

\bibitem[Lev1]{Levin4}
G.~M. Levin.
\newblock Bifurcation set of parameters of a family of quadratic mappings.
\newblock In {\em Approximate methods for investigating differential equations
  and their applications}, pages 103--109. Ku\u\i byshev. Gos. Univ.,
  Kuybyshev, 1982.

\bibitem[Lev2]{Levin1}
Genadi Levin.
\newblock On the theory of iterations of polynomial families in the complex
  plane.
\newblock {\em J. Soviet Math.}, 52(6):3512--3522, 1990.

\bibitem[Lev3]{Levin2}
Genadi Levin.
\newblock Multipliers of periodic orbits in spaces of rational maps.
\newblock {\em Ergodic Theory Dynam. Systems}, 31(1):197--243, 2011.

\bibitem[Lev4]{Levin3}
Genadi Levin.
\newblock Perturbations of weakly expanding critical orbits, 2011.
\newblock preprint arXiv math.DS/1111.6270.

\bibitem[Lev5]{Levy}
Alon Levy.
\newblock An algebraic proof of thurston's rigidity for a polynomial, 2012.
\newblock preprint arXiv : math.AG/1201.1969.

\bibitem[M]{McMullen}
Curtis~T. McMullen.
\newblock {\em Complex dynamics and renormalization}, volume 135 of {\em Annals
  of Mathematics Studies}.
\newblock Princeton University Press, Princeton, NJ, 1994.

\bibitem[P]{Poirier}
Alfredo Poirier.
\newblock On postcritically finite polynomials, part 2: Hubbard trees., 1993.
\newblock preprint of the IMS -- ims93-7.

\bibitem[S]{Silverman}
Joseph~H. Silverman.
\newblock {\em The arithmetic of dynamical systems}, volume 241 of {\em
  Graduate Texts in Mathematics}.
\newblock Springer, New York, 2007.

\bibitem[Y]{Yuan}
Xinyi Yuan.
\newblock Big line bundles over arithmetic varieties.
\newblock {\em Invent. Math.}, 173(3):603--649, 2008.

\end{thebibliography}

\end{document}